\documentclass[reqno, 12pt]{amsart}

\usepackage{amsmath, amsthm, amssymb, graphicx, paralist, cite, color, relsize}
\usepackage[all]{xypic}
\definecolor{e-mail}{rgb}{0,.40,.80}
\definecolor{reference}{rgb}{.20,.60,.22}
\definecolor{citation}{rgb}{0,.40,.80}
\usepackage[pdfstartview=FitH, colorlinks, linkcolor=reference, citecolor=citation, urlcolor=e-mail]{hyperref}

\newtheorem{thm}{Theorem}
\newtheorem*{mthm}{Main Theorem}
\newtheorem*{thm*}{Theorem~\protect{\ref{thm:main}}}
\newtheorem{cor}[thm]{Corollary}
\newtheorem{lem}[thm]{Lemma}
\newtheorem{prop}[thm]{Proposition}

\theoremstyle{definition}
\newtheorem{defn}[thm]{Definition}
\theoremstyle{remark}
\newtheorem{rem}[thm]{Remark}
\numberwithin{thm}{section}
\theoremstyle{definition}

\theoremstyle{definition}

\newtheorem{eg}[thm]{Example}
\theoremstyle{definition}
\numberwithin{equation}{section}

\newcommand{\K}{\mathbb K} 

\title{Mahler discrete residues and summability for rational functions}

\author{Carlos E. Arreche}
\address{Department of Mathematical Sciences, The University of Texas at Dallas, Richardson, TX 75080}
\email{arreche@utdallas.edu}
 \thanks{The work of C.E.~Arreche was partially supported by NSF grant CCF-1815108.} 

\author{Yi Zhang}
\address{Department of Foundational  Mathematics, School of Science, Xi'an Jiaotong-Liverpool University,
 Suzhou, 215123, China}
\email{Yi.Zhang03@xjtlu.edu.cn}
\thanks{Y.~Zhang was supported by XJTLU Research Development Funding No.\ RDF-20-01-12, the NSFC Young Scientist Fund No.\ 12101506, and the Natural Science Foundation of the Jiangsu Higher Education Institutions of China No.\ 21KJB110032.}

\date{\today}

\begin{document}

\begin{abstract}
We construct Mahler discrete residues for rational functions and show that they comprise a complete obstruction to the Mahler summability problem of deciding whether a given rational function $f(x)$ is of the form $g(x^p)-g(x)$ for some rational function $g(x)$ and an integer $p > 1$. This extends to the Mahler case the analogous notions, properties, and applications of discrete residues (in the shift case) and $q$-discrete residues (in the $q$-difference case) developed by Chen and Singer. Along the way we define several additional notions that promise to be useful for addressing related questions involving Mahler difference fields of rational functions, including in particular telescoping problems and problems in the (differential) Galois theory of Mahler difference equations.
\end{abstract}

\keywords{Mahler operator, difference fields, difference equations, partial fractions, discrete residues, summability, creative telescoping}

\maketitle


\section{Introduction} \label{sec:introduction}
Continuous residues are fundamental tools in complex analysis, and have extensive and compelling applications in combinatorics~\cite{Flajolet_Sedgewick}. In the last decade, a theory of ($q$-)discrete residues was proposed in~\cite{ChenSinger2012} for the study of telescoping problems, which has found essential applications in several other closely related problems (see \cite{HouWang:2015,arreche:2017,Chen:2018,ArrecheZhang2020} for some examples). A theory of residues for skew rational functions was developped in \cite{Caruso2021}, which has applications in duals of linearized Reed-Solomon codes \cite{CarusoDurand2021}. The elliptic orbit residues defined in \cite{HardouinSinger2021} have applications in the combinatorial study of walks in the quarter plane. We propose here a theory of Mahler discrete residues aimed at bringing to the Mahler case the successes of these earlier notions of residues.

Let $\K$ be a field of characteristic zero and $\K(x)$ be the field of rational functions in an indeterminate $x$ over $\K$. Fix an integer $p\geq 2$. We study the \emph{Mahler summability problem for rational functions}: given $f(x)\in\mathbb{K}(x)$, decide effectively whether $f(x)=g(x^p)-g(x)$ for some $g(x)\in\mathbb{K}(x)$; if so, we say $f(x)$ is \emph{Mahler summable}. 

The motivation to study Mahler equations comes from several directions: they find applications in automata theory (automatic sequences), transcendence, and number theory, to name a few. We refer to~\cite{Dreyfus2018} for more details, and also for an altogether different approach to Mahler summability: the algorithm of \cite[\S3]{Dreyfus2018} computes all the rational solutions to any linear Mahler equation. Thus with this one can decide, in particular, whether a \emph{given} $f(x)\in\mathbb{K}(x)$ is Mahler summable by computing (or showing non-existence of) a certificate $g(x)\in\mathbb{K}(x)$ such that $f(x)=g(x^p)-g(x)$.

Our goal here is different: we wish to construct a complete obstruction to Mahler summability. Let us elaborate. The image of the $\mathbb{K}$-linear map $\Delta: g(x)\mapsto g(x^p)-g(x)$ is the kernel of some other $\mathbb{K}$-linear map (call it $\nabla$) --- but what is it? Determining such a $\nabla$ explicitly is algorithmically desirable because it allows to decide the Mahler summability of $f(x)\in\mathbb{K}(x)$ without computing the certificate $g(x)\in\mathbb{K}(x)$, whose computation is often in practice both expensive and not strictly necessary (cf.~\cite[\S1 \& Table~1]{BCCL:2010}, \cite[\S1 \& Table 1]{BCCLX:2013}, and \cite[\S1]{BLS:2013}). We construct such a $\nabla$ explicitly in Section~\ref{sec:reduction}, in terms of our new notion of \emph{Mahler discrete residues} for rational functions, and prove in Section~\ref{sec:proof}:

\begin{mthm}\label{thm:main} $f(x)\in\mathbb{K}(x)$ is Mahler summable if and only if all of the Mahler discrete residues of $f$ are zero.
\end{mthm}

The discrete and $q$-discrete residues developed in \cite{ChenSinger2012} comprise complete obstructions to the analogous summability problems for $f(x)\in\mathbb{K}(x)$, of deciding whether there exists $g(x)\in\mathbb{K}(x)$, such that $f(x)=g(x+1)-g(x)$, or such that $f(x)=g(qx)-g(x)$, for $q\in\mathbb{K}$ neither zero nor a root of unity. This theoretical property of ($q$-)discrete residues is precisely what enables their applications to the telescoping problems considered in \cite{ChenSinger2012} and their indispensable role in the development of the algorithms in \cite{arreche:2017, ArrecheZhang2020}. We envision analogous applications of Mahler discrete residues to telescoping problems and in the development of algorithms to compute (differential) Galois groups for Mahler difference equations.
 
Our strategy is inspired by that of \cite{ChenSinger2012} (but see Remark~\ref{rem:mahler-phenomena}): we utilize the coefficients in the partial fraction decomposition of $f(x)$ to construct an aspiring certificate $g(x)\in\mathbb{K}(x)$ such that \begin{equation}\label{eq:mahler-reduction}\bar{f}(x):=f(x)+\bigl(g(x^p)-g(x)\bigr)\end{equation} is Mahler summable if and only if $\bar{f}(x)=0$. The Mahler discrete residues of $f(x)$ are (vectors whose components are) the coefficients occurring in the partial fraction decomposition of $\bar{f}(x)$. This $\bar{f}(x)$ plays the role of a \emph{Mahler remainder} of $f(x)$, analogous to the remainder of Hermite reduction in the context of integration.


\section{Preliminaries} \label{sec:preliminaries}
Here we define the notation and conventions used throughout this work, and prove some ancillary results. We fix once and for all an algebraically closed field $\mathbb{K}$ of characteristic zero and an integer $p\geq 2$ (not necessarily prime). We denote by $\mathbb{K}(x)$ the field of rational functions in an indeterminate $x$ over $\mathbb{K}$. We often suppress the functional notation and write simply $f\in\mathbb{K}(x)$ instead of $f(x)$.

\begin{defn}\label{defn:basic}We denote by $\sigma:\mathbb{K}(x)\rightarrow \mathbb{K}(x)$ the $\mathbb{K}$-linear endomorphism defined by $\sigma(x)=x^p$, called the \emph{Mahler operator}, so that $\sigma(f(x))=f(x^p)$ for $f(x)\in\mathbb{K}(x)$. We write $\Delta:=\sigma-\mathrm{id}$, so that $\Delta(f(x))=f(x^p)-f(x)$ for $f(x)\in\mathbb{K}(x)$.

We say that $f\in \mathbb{K}(x)$ is \emph{Mahler summable} if $f=\Delta(g)$ for some $g\in\mathbb{K}(x)$. The \emph{Mahler summability problem for rational functions} is: given $f\in\mathbb{K}(x)$, decide whether $f$ is Mahler summable.
\end{defn}

Let $\mathbb{K}^\times=\mathbb{K}\backslash\{0\}$ denote the multiplicative group of $\mathbb{K}$. Let $\mathbb{K}^\times_t$ denote the torsion subgroup of $\mathbb{K}^\times$, i.e., the group of roots of unity in $\mathbb{K}^\times$. For $\zeta\in\mathbb{K}^\times_t$, the \emph{order} of $\zeta$ is the smallest $r\in\mathbb{N}$ such that $\zeta^r=1$. We fix once and for all a \emph{compatible system} of $p$-power roots of unity $(\zeta_{p^n})_{n\geq 0} \subset \mathbb{K}^\times_t$, that is, each $\zeta_{p^n}$ has order $p^n$ and $\smash{\zeta_{p^n}^{p^\ell}=\zeta_{p^{n-\ell}}}$ for $0\leq\ell\leq n$. We denote by $\pi^n_\ell:\mathbb{Z}/p^n\mathbb{Z}\twoheadrightarrow\mathbb{Z}/p^\ell\mathbb{Z}$ and by $\pi_n:\mathbb{Z}\twoheadrightarrow\mathbb{Z}/p^n\mathbb{Z}$ the canonical projections.

Each $f\in\mathbb{K}(x)$ decomposes uniquely as \begin{equation}\label{eq:f-decomposition}f=f_L+f_T,\quad\text{where}\end{equation} $f_L\in\mathbb{K}[x,x^{-1}]$ is a Laurent polynomial and $f_T=\frac{a}{b}$ for polynomials $a,b\in\mathbb{K}[x]$ such that $b\neq 0$ and, either $a=0$, or else $\mathrm{deg}(a)<\mathrm{deg}(b)$ and $\mathrm{gcd}(a,b)=1=\mathrm{gcd}(x,b)$. The subscript $L$ stands for ``Laurent''. The subscript $T$ stands for ``Tree'' (see Definition~\ref{defn:trees}).

\begin{lem}\label{lem:rational-decomposition} The decomposition $\mathbb{K}(x)\simeq \mathbb{K}[x,x^{-1}]\oplus \mathbb{K}(x)_T$ given by $f\leftrightarrow f_L\oplus f_T$ as in \eqref{eq:f-decomposition} is $\sigma$-stable. For $f,g\in\mathbb{K}(x)$, $f=\Delta(g)$ if and only if $f_L=\Delta(g_L)$ and $f_T=\Delta(g_T)$.
\end{lem}

\begin{proof} We see that $\sigma(f_L)\in\mathbb{K}[x,x^{-1}]$ for any $f_L\in\mathbb{K}[x,x^{-1}]$. By the Euclidean algorithm, $\mathrm{gcd}(\sigma(a),\sigma(b))=\sigma(\mathrm{gcd}(a,b))$ for any $0\neq a,b\in\mathbb{K}[x]$. Thus the $\mathbb{K}$-subspace $\mathbb{K}(x)_T$ is also stabilized by $\sigma$. It follows that $\Delta(g)=\Delta(f)$ if and only if $\Delta(g_L)=f_L$ and $\Delta(g_T)=f_T$, for any $f,g\in\mathbb{K}(x)$
\end{proof}


\subsection{Mahler trajectories, trees, and cycles}\label{sec:trajectories-trees-cycles}

We let $\mathcal{P}:=\{p^n \ | \ n\in\mathbb{Z}_{\geq 0}\}$ denote the multiplicative monoid of non-negative powers of $p$. Then $\mathcal{P}$ acts on $\mathbb{Z}$ by multiplication, and the set of \emph{maximal trajectories} for this action is \[\mathbb{Z}/\mathcal{P} := \bigl\{\{0\}\bigr\} \cup \bigl\{\{ip^n \ | \ n\in\mathbb{Z}_{\geq 0}\} \ \big| \ i\in\mathbb{Z} \ \text{such that} \ p\nmid i \bigr\}.\]

\begin{rem}The usage of \emph{trajectory} is perhaps unfamiliar to some readers: it is standard in the context of monoid (and more generally semigroup) actions, and replaces the more familiar notion of \emph{orbit} for group actions. As in that more familiar setting, the elements $\theta\in\mathbb{Z}/\mathcal{P}$ are pairwise disjoint sets whose union is all of $\mathbb{Z}$.\end{rem}

\begin{defn}\label{defn:trajectory-projection}
For a maximal trajectory $\theta\in\mathbb{Z}/\mathcal{P}$, the $\theta$\emph{-subspace} \begin{equation}\label{eq:trajectory-subspace}  \mathbb{K}[x,x^{-1}]_\theta := \left\{\sum_j c_j x^j\in\mathbb{K}[x,x^{-1}] \ \middle| \ c_j=0 \ \text{for all} \ j \notin \theta\right\}.\end{equation} The $\theta$\emph{-component} $f_\theta$ of $f\in\mathbb{K}(x)$ is the projection of the component $f_L$ of $f$ in \eqref{eq:f-decomposition} to the $\theta$-subspace $\mathbb{K}[x,x^{-1}]_\theta$ in \eqref{eq:trajectory-subspace}.\end{defn}

\begin{lem}\label{lem:trajectory-projection} For $f,g\in\mathbb{K}(x)$, $f_L=\Delta(g_L)$ if and only if $f_\theta=\Delta(g_\theta)$ for every $\theta\in\mathbb{Z}/\mathcal{P}$.\end{lem}

\begin{proof} This follows by observing that the $\mathbb{K}$-decomposition $\mathbb{K}[x,x^{-1}]  \hspace{.155em}\simeq  \hspace{.155em}\bigoplus_{\theta\in\,\mathbb{Z}/\mathcal{P} }\mathbb{K}[x,x^{-1}]_\theta$ is $\sigma$-stable (cf.~\cite[\S5]{CHLW:2016}).
\end{proof}

\begin{defn} \label{defn:trees} We denote by $\mathcal{T}_M$ the set of equivalence classes in $\mathbb{K}^\times$ for the equivalence relation $\alpha\sim \gamma\Leftrightarrow\alpha^{p^r}=\gamma^{p^s}$ for some $r,s\in\mathbb{Z}_{\geq 0}$. For $\alpha\in\mathbb{K}^\times$, we denote by $\tau(\alpha)\in\mathcal{T}_M$ the equivalence class of $\alpha$ under $\sim$. The elements $\tau\in\mathcal{T}_M$ are called \emph{Mahler trees}.
\end{defn}

\begin{rem} \label{rem:terminology} 
The usage of \emph{tree} in Definition~\ref{defn:trees} is motivated by the fact that one can define a digraph structure $D(\tau(\alpha))$ on the vertex set $\tau(\alpha)$ with an edge from $\xi$ to $\gamma$ whenever $\xi^p=\gamma$, whose underlying (undirected) graph is connected and acyclic provided that $\alpha\notin\mathbb{K}^\times_t$. We find the terminology useful and suggestive even when $\alpha \in\mathbb{K}^\times_t$, because even in this exceptional case we do obtain a tree after collapsing the unique cycle in $D(\tau)$ defined below.
\end{rem}

\begin{defn} \label{defn:cycles} For a Mahler tree $\tau\in\mathcal{T}_M$, the \emph{Mahler cycle} of $\tau$ is \[\mathcal{C}(\tau):=\{\gamma \in \tau \ | \ \gamma \ \text{is a root of unity of order coprime to } p\}.\] The \emph{cycle length} of $\tau$ is defined to be $e(\tau):=|\mathcal{C}(\tau)|$.
\end{defn}

\begin{eg} \label{eg:trees} (Cf.~\cite[Figures~4 and 5]{Dreyfus2018}). Let us illustrate the definitions of Mahler trees and Mahler cycles with $\mathbb{K}=\mathbb{C}$ and $p=3$. In this example we write $\zeta_n:=e^{\frac{2\pi \sqrt{-1}}{n}}\in\mathbb{C}^\times$, for concreteness.

The vertices in the digraph $D(\tau(2))$ near $\alpha=2$ are:
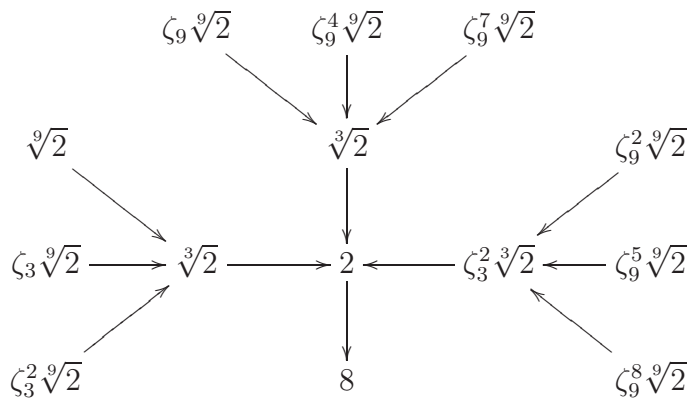
\[\xymatrix@=.7em{ & &  \zeta_9\sqrt[9]{2} \ar[rrdd] & &  \zeta_9^4\sqrt[9]{2} \ar[dd] & &  \zeta_9^7\sqrt[9]{2}\ar[ddll] \\ \\ 
\sqrt[9]{2} \ar[ddrr] & &  & &  \sqrt[3]{2}\ar[dd]  & &  & &  \zeta_9^2\sqrt[9]{2}\ar[ddll] \\ \\ 
\zeta_3\sqrt[9]{2}\ar[rr] & &  \sqrt[3]{2} \ar[rr]& &  2 \ar[dd]& &  \zeta_3^2\sqrt[3]{2} \ar[ll]& &  \zeta_9^5\sqrt[9]{2} \ar[ll]\\ \\ 
\zeta_3^2\sqrt[9]{2} \ar[uurr] & &  & &  8 & &  & &  \zeta_9^8\sqrt[9]{2}\ar[uull]}
\]

For $\alpha=\zeta_4$, we have $\mathcal{C}(\tau(\zeta_4))=\left\{\zeta_4,\zeta_4^3\right\}$, so the cycle length $e(\tau(\zeta_4))=2$. The vertices in the digraph $D(\tau(\zeta_4))$ near $\mathcal{C}(\tau(\zeta_4))$ are:

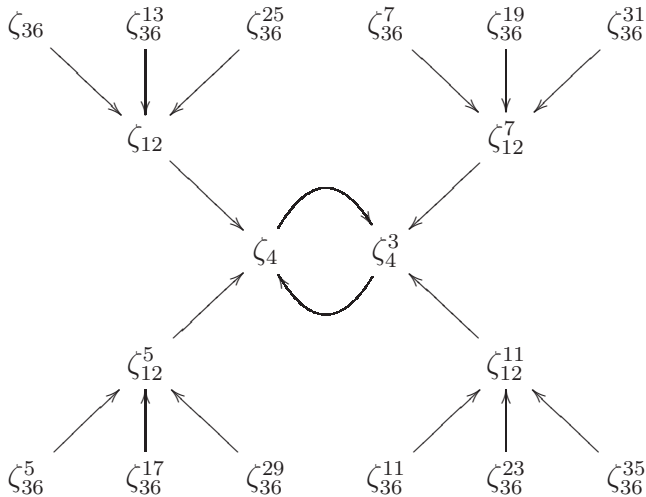
\[\xymatrix@=.7em{ \zeta_{36} \ar[ddrr] & & \zeta_{36}^{13} \ar[dd] & & \zeta_{36}^{25} \ar[ddll]& & \zeta_{36}^7 \ar[ddrr]& & \zeta_{36}^{19} \ar[dd]& & \zeta_{36}^{31} \ar[ddll] \\ \\ 
& & \zeta_{12} \ar[ddrr]& & & & & & \zeta_{12}^7\ar[ddll] \\ \\ 
& & & & \zeta_4 \ar@/^2pc/[rr] & & \zeta_4^3 \ar@/^2pc/[ll] \\ \\ 
& & \zeta_{12}^5 \ar[uurr]& & & & & & \zeta_{12}^{11} \ar[uull]\\ \\ 
\zeta_{36}^5 \ar[uurr]& & \zeta_{36}^{17} \ar[uu]& & \zeta_{36}^{29} \ar[uull]& & \zeta_{36}^{11} \ar[uurr]& & \zeta_{36}^{23} \ar[uu]& & \zeta_{36}^{35}\ar[uull]
}
\]
\end{eg}

\begin{rem}\label{rem:cycle-facts} Let us collect some immediate observations about Mahler cycles that we shall use, and refer to, throughout the sequel.

For a Mahler tree $\tau\in\mathcal{T}_M$ it follows from the Definition~\ref{defn:trees} that either $\tau\subset\mathbb{K}^\times_t$ or else $\tau\cap\mathbb{K}^\times_t=\emptyset$. In particular, $\mathcal{C}(\tau)=\emptyset\Leftrightarrow e(\tau)=0$, which occurs precisely when $\tau\not\subset\mathbb{K}^\times_t$ (the \emph{non-torsion case}).

On the other hand, $\mathbb{K}^\times_t$ consists of the pre-periodic points for the action of the monoid $\mathcal{P}$ on $\mathbb{K}^\times$ given by $\alpha\mapsto \alpha^{p^n}$ for $n\in\mathbb{Z}_{\geq 0}$. For $\tau\subset\mathbb{K}^\times_t$ (the \emph{torsion case}), the Mahler cycle $\mathcal{C}(\tau)$ is a non-empty set endowed with a simply transitive action of the quotient monoid $\mathcal{P}/\mathcal{P}^e\simeq\mathbb{Z}/e\mathbb{Z}$, where $\mathcal{P}^e:=\{p^{ne} \ | \ n\in\mathbb{Z}\}$, and $e:=e(\tau)$. We emphasize that in general $\mathcal{C}(\tau)$ is only a set, and not a group. The Mahler tree $\tau(1)$ consists precisely of the roots of unity $\zeta\in\mathbb{K}^\times_t$ whose order $r$ is such that $\mathrm{gcd}(r,p^n)=r$ for some $p^n\in\mathcal{P}$, or equivalently such that every prime factor of $r$ divides $p$. When $\tau\subset\mathbb{K}^\times_t$ but $\tau\neq \tau(1)$, the cycle length $e(\tau)$ coincides with the order of $p$ in the group of units $(\mathbb{Z}/r\mathbb{Z})^\times$, where $r>1$ is the common order of the roots of unity $\gamma\in\mathcal{C}(\tau)$, and for any given $\gamma\in\mathcal{C}(\tau)$ we have that $\mathcal{C}(\tau)=\{\gamma^{p^\ell} \ | \ 0\leq \ell \leq e-1\}$.
\end{rem}


\subsection{Mahler supports and singular supports}\label{sec:forest-support}

Mahler trees allow us to define the following bespoke variants of the singular support $\mathrm{sing}(f)$ of a rational function $f$ (i.e., its set of poles), which are particularly well-suited to the Mahler context.

\begin{defn} \label{defn:supp} For $f\in \mathbb{K}(x)$, we define $\mathrm{supp}(f)\subset \mathcal{T}_M\cup\{\infty\}$, called the \emph{Mahler support} of $f$, as follows:
\begin{itemize}
\item $\infty\in\mathrm{supp}(f)$ if and only if $f_L\neq 0$; and
\item for $\tau\in\mathcal{T}_M$, $\tau\in\mathrm{supp}(f)$ if and only if $\tau$ contains a pole of $f$.
\end{itemize}

For $\tau\in\mathcal{T}_M$, the \emph{singular support} of $f$ in $\tau$, denoted by $\mathrm{sing}(f,\tau)$, is the (possibly empty) set of poles of $f$ contained in $\tau$.
\end{defn}

We omit the straightforward proof of the following lemma.

\begin{lem}\label{lem:stable-support} For $f,g\in\mathbb{K}(x)$ and $0\neq c\in\mathbb{K}$ we have the following:
\begin{enumerate}
\item $\mathrm{supp}(f)=\emptyset\Leftrightarrow f=0$;
\item $\mathrm{supp}(\sigma(f))=\mathrm{supp}(f)=\mathrm{supp}(c\cdot f)$; and
\item $\mathrm{supp}(f+g)\subseteq \mathrm{supp}(f)\cup\mathrm{supp}(g)$.
\end{enumerate}
\end{lem}

%
%

\begin{defn}\label{defn:tree-projection}
For a Mahler tree $\tau\in\mathcal{T}_M$, the $\tau$\emph{-subspace} \begin{equation}\label{eq:tree-subspace}\mathbb{K}(x)_\tau := \bigl\{f_T\in \mathbb{K}(x)_T \ \big| \ \mathrm{supp}(f_T)\subseteq \left\{\tau\right\}\bigr\}.\end{equation} For $f\in\mathbb{K}(x)$, the $\tau$\emph{-component} $f_\tau$ of $f$ is the projection of the component $f_T$ of $f$ in \eqref{eq:f-decomposition} to the $\tau$-subspace $\mathbb{K}(x)_\tau$ in \eqref{eq:tree-subspace}.\end{defn}

\begin{lem}\label{lem:tree-decomposition} For $f,g\in\mathbb{K}(x)$, $f_T=\Delta(g_T)$ if and only if $f_\tau =\Delta(g_\tau )$ for every $\tau\in\mathcal{T}_M$.
\end{lem}

\begin{proof}
It follows from Lemma~\ref{lem:stable-support} that the $\mathbb{K}$-linear decomposition $\mathbb{K}(x)_T\simeq\bigoplus_{\tau\in\,\mathcal{T}_M}\mathbb{K}(x)_\tau$ is $\sigma$-stable (cf.~\cite[\S5]{CHLW:2016}).
\end{proof}


\subsection{Mahler dispersion}

We now define a Mahler variant of the notion of (polar) dispersion used in \cite{ChenSinger2012}, following the original definitions in \cite{Abramov:1971,Abramov:1974}.

\begin{defn} \label{defn:dispersion} For $f\in \mathbb{K}(x)$ and $\tau\in\mathrm{supp}(f)$, the \emph{Mahler dispersion} of $f$ at $\tau$, denoted by $\mathrm{disp}(f,\tau)$, is defined as follows.

If $\tau\in \mathcal{T}_M$, $\mathrm{disp}(f,\tau)$ is the largest $d\in\mathbb{Z}_{\geq 0}$ (if it exists) for which there exists $\alpha\in\mathrm{sing}(f,\tau)$ such that $\alpha^{p^d}\in\mathrm{sing}(f,\tau)$. If there is no such $d\in\mathbb{Z}_{\geq 0}$, then we set $\mathrm{disp}(f,\tau)=\infty$.

If $\tau=\infty$, let us write $f_L=\sum_{i=n}^Nc_ix^i\in\mathbb{K}[x,x^{-1}]$ with $c_nc_N\neq 0$.
\begin{itemize}
\item If $f_L=c_0\neq 0$ then we set $\mathrm{disp}(f,\infty)=0$; otherwise
\item $\mathrm{disp}(f,\infty)$ is the largest $d\in\mathbb{Z}_{\geq 0}$ for which there exists an index $i\neq 0$ such that $c_i\neq 0$ and $c_{ip^d}\neq 0$.
\end{itemize} 
For $f\in\mathbb{K}(x)$ and $\tau\in\mathcal{T}_M\cup\{\infty\}$ such that $\tau\notin\mathrm{supp}(f)$, we do not define $\mathrm{disp}(f,\tau)$ at all (cf.~\cite{Abramov:1971,Abramov:1974,ChenSinger2012}).
\end{defn}

Similarly as in the shift and $q$-difference cases (cf.~\cite[Lemma 6.3]{HardouinSinger2008} and \cite[Lemma~2.4 and Lemma~2.9]{ChenSinger2012}), Mahler dispersions will play a crucial role in what follows. As we prove in Corollary~\ref{cor:summable-dispersion}, they already provide a partial obstruction to summability: if $f\in\mathbb{K}(x)$ is Mahler summable then every Mahler dispersion of $f$ is non-zero. Moreover, Mahler dispersions also detect whether $f$ has any ``bad'' poles (i.e., at roots of unity of order coprime to $p$) according to:

\begin{lem} \label{lem:infinite-dispersion} Let $f\in\mathbb{K}(x)$ and $\tau\in\mathrm{supp}(f)$. Then $\mathrm{disp}(f,\tau)=\infty$ if and only if $\mathrm{sing}(f,\tau)\cap\mathcal{C}(\tau)\neq \emptyset$.
\end{lem}

\begin{proof}($\Rightarrow$). If $\mathrm{disp}(f,\tau)=\infty$, then there exist $\alpha, \gamma \in \mathrm{sing}(f,\tau)$ such that $\alpha^{p^d}=\gamma$ for infinitely many $d\in\mathbb{Z}_{\geq 0}$. Thus, both $\gamma$ and $\alpha$ are roots of unity. Let $r$ be the order of $\alpha$. For any $d\in\mathbb{Z}_{\geq 0}$, the order of $\alpha^{p^d}$ is $r_d:=r/\mathrm{gcd}(p^d,r)$, and we see that for every sufficiently large $d$, the order $r_d$ of $\alpha^{p^d}$ is coprime to $p$, and therefore $\gamma$ is a root of unity of order coprime to $p$.

\noindent ($\Leftarrow$).~For $\gamma\in \mathrm{sing}(f,\tau)\cap\mathcal{C}(\tau)$ we have $\gamma^{p^{e(\tau)\cdot n}  }=\gamma$ for every $n\in\mathbb{Z}_{\geq 0}$ (cf.~Remark~\ref{rem:cycle-facts}), whence $\mathrm{disp}(f,\tau)=\infty$ by Definition~\ref{defn:dispersion}.
\end{proof}


\subsection{Mahler coefficients for partial fractions}\label{sec:mahler-coefficients}

We now study the effect of the Mahler operator $\sigma$ on partial fraction decompositions. For $\alpha\in\mathbb{K}^\times$, $m\in\mathbb{N}$, and $1\leq k \leq m$, we define the \emph{Mahler coefficients} $V^m_k(\alpha)\in\mathbb{K}$ implicitly by \begin{equation}\label{eq:mahler-coefficients} \sigma\left(\frac{1}{(x-\alpha^p)^m}\right)=\frac{1}{(x^p-\alpha^p)^m}=\sum_{k=1}^m\sum_{i=0}^{p-1}\frac{V^m_k(\zeta_p^i\alpha)}{(x-\zeta_p^i\alpha)^k}.\end{equation} These coefficients are computed explicitly with the following result.

\begin{lem}\label{lem:mahler-coefficients} There exist \emph{universal coefficients} $\mathbb{V}^m_k\in\mathbb{Q}$ such that \[V^m_k(\alpha)=\mathbb{V}^m_k\cdot\alpha^{k-pm}\] for every $\alpha\in\mathbb{K}^\times$ and $1\leq k\leq m$. Moreover, these $\mathbb{V}^m_k$ are the first $m$ Taylor coefficients at $x=1$ of \begin{equation}\label{eq:taylor-coefficients}(x^{p-1}+\dots+x+1)^{-m}=\sum_{k=1}^m\mathbb{V}^m_k\cdot(x-1)^{m-k}+O((x-1)^m).\end{equation}\end{lem}

\begin{proof} We claim that $V^m_k(\alpha)=V^m_k(1)\cdot\alpha^{k-pm}$ for every $\alpha\in\mathbb{K}^\times$. To see this, set $x=\alpha y$ for a new indeterminate $y$, and note that \begin{multline*}\sum_{k=1}^m\sum_{i=0}^{p-1}\frac{V^m_k(\zeta_p^i\alpha)}{(x-\zeta_p^i\alpha)^k}=\frac{1}{(x^p-\alpha^p)^m}=\alpha^{-pm}\cdot\frac{1}{(y^p-1)^m} = \\ =\alpha^{-pm}\sum_{k=1}^m\sum_{i=0}^{p-1}\frac{V^m_k(\zeta_p^i)}{(y-\zeta_p^i)^k} =\alpha^{-pm}\sum_{k=1}^m\sum_{i=0}^{p-1}\frac{V^m_k(\zeta_p^i)\alpha^k}{(x-\zeta_p^i\alpha)^k}.\end{multline*} It follows that $V^m_k(\zeta_p^i\alpha)=V^m_k(\zeta_p^i)\alpha^{k-pm}$ for $i=0,\dots,p-1$. In particular for $i=0$ we obtain $V^m_k(\alpha)=V^m_k(1)\alpha^{k-pm}$, as claimed. Setting $\mathbb{V}^m_k:=V^m_k(1)$, we see from \eqref{eq:mahler-coefficients} that $\mathbb{V}^m_k$ is the usual continuous residue of order $k$ of $f(x):=(x^p-1)^{-m}$. The formula~\eqref{eq:taylor-coefficients} follows from \cite[Section~2]{Bronstein:1991}, where it is shown that $\mathbb{V}^m_k=\frac{g^{(m-k)}(1)}{(m-k)!}$, where $g(x):=(x-1)^m f(x)=(x^{p-1}+\dots+1)^{-m}$.\end{proof}

The following immediate consequence of Lemma~\ref{lem:mahler-coefficients} is obtained by evaluating \eqref{eq:taylor-coefficients} at $x=1$.

\begin{cor}\label{cor:top-coefficient} For $\alpha\in\mathbb{K}^\times$, $V^m_m(\alpha)=p^{-m}\alpha^{m-pm}$.\end{cor}

\begin{rem}\label{rem:mahler-phenomena}We see in \eqref{eq:mahler-coefficients} two phenomena that arise in the Mahler context and have no counterpart in the shift and $q$-dilation settings considered in \cite{ChenSinger2012} --- the main inspiration for the present work. Let $f\in\mathbb{K}(x)$ such that $0\neq f_T$ as in \eqref{eq:f-decomposition}. Then: 
\begin{enumerate}
\item the number of poles of $\sigma(f_T)$ (counted either with or without multiplicity!) is strictly larger than that of $f_T$; and
\item the (classical/continuous) higher-order residues of $f_T$ ``leak'' into the lower-order residues of $\sigma(f_T)$.
\end{enumerate}
These two phenomena are mainly responsible for our need to create new and somewhat intricate bookkeeping devices in the Mahler setting, which were (invisibly) not necessary in the shift and $q$-dilation settings considered in \cite{ChenSinger2012}, in order to develop our proposed analogous theory of Mahler discrete residues. \end{rem}


\begin{eg}\label{eg:coefficients} Let us illustrate the definition of Mahler coefficients with $p=3$, $m=2$, and $\alpha^3=1$. Then \eqref{eq:mahler-coefficients} becomes \[\sigma\left(\frac{1}{(x-1)^{2}}\right)=\frac{1}{(x^3-1)^{2}}=\sum_{k=1}^2\sum_{i=0}^2\frac{\mathbb{V}^2_k\cdot\zeta_3^{ki}}{(x-\zeta_3^i)^k},\] since, according to Lemma~\ref{lem:mahler-coefficients}, $V^2_k(\zeta_3^i)=\mathbb{V}^2_k\cdot(\zeta_3^i)^{k-6}=\mathbb{V}^2_k\cdot\zeta_3^{ki}$ for $k=1,2$. We find in this case, using \eqref{eq:taylor-coefficients} in Lemma~\ref{lem:mahler-coefficients}, that \[\mathbb{V}^2_2 =\bigl(x^2+x+1\bigr)^{-2}\Big|_{x=1}  =\frac{1}{9}; \ \ \text{and} \ \ 
\mathbb{V}^2_1 =\bigl((x^2+x+1)^{-2}\bigr)'\Big|_{x=1}  =-\frac{2}{9}.
\]
One can verify using a computer algebra system (or by hand!) that the partial fraction decomposition of $9\cdot (x^3-1)^{-2}$ is indeed  \[ \frac{1}{(x-1)^2}+\frac{\zeta_3^2}{(x-\zeta_3)^2}+\frac{\zeta_3}{(x-\zeta_3^2)^2}+\frac{-2}{x-1}+\frac{-2\zeta_3}{x-\zeta_3}+\frac{-2\zeta_3^2}{x-\zeta_3^2}.\]
\end{eg}


\section{Mahler dispersion and summability}\label{sec:summable-dispersion}

The goal of this section is to prove Corollary~\ref{cor:summable-dispersion}: if $f\in\mathbb{K}(x)$ is Mahler summable then $\mathrm{disp}(f,\tau)\neq 0$ for every $\tau\in\mathrm{supp}(f)$. This is an essential ingredient in our proofs, following \cite{ChenSinger2012}. The following result is a Mahler analogue of \cite[Lemma 2.6]{CFMS:2021}. 

\begin{prop}\label{prop:summable-dispersion} Let $f,g\in\mathbb{K}(x)$ such that $f=\Delta(g)$. Then $\mathrm{supp}(f)\subseteq\mathrm{supp}(g)$. Moreover, $\mathrm{disp}(f,\tau)=\mathrm{disp}(g,\tau)+1$ for every $\tau\in\mathrm{supp}(f)$,
with the convention that $\infty+1=\infty$.
\end{prop}

\begin{proof} By Lemma~\ref{lem:stable-support}, $\tau\in\mathrm{supp}(f) \Rightarrow \tau\in\mathrm{supp}(g)$. We consider separately the two main cases: (1)~$\tau=\infty$; and (2)~$\tau\in\mathcal{T}_M$.

(1).~For $f_L,g_L\in\mathbb{K}[x,x^{-1}]$ as in \eqref{eq:f-decomposition}, we have $0\neq f_L=\Delta(g_L)$, since $\infty\in\mathrm{supp}(f)$, and in particular $g_L\notin\mathbb{K}$. Then $f_\theta=\Delta(g_\theta)$ for each $\theta\in \mathbb{Z}/\mathcal{P}$ by Lemma~\ref{lem:trajectory-projection}. Since, for $\theta=\{0\}$,  $f_{\{0\}}=\Delta(g_{\{0\}})=0$, it follows from Definition~\ref{defn:dispersion} that \[\mathrm{disp}(f,\infty)=\mathrm{max}\left\{\mathrm{disp}\left(f_\theta,\infty\right) \ \middle| \ \{0\}\neq \theta\in\mathbb{Z}/\mathcal{P}, \ f_\theta\neq 0\right\}.\] We claim $\mathrm{disp} \left(\Delta\left(g_\theta\right) , \infty\right)=\mathrm{disp} \left(g_\theta, \infty\right)+1$ for every $g_\theta\in\mathbb{K}[x,x^{-1}]_\theta$ with $\{0\}\neq\theta\in\mathbb{Z}/\mathcal{P}$, which will conclude the proof of (1). To prove the claim, let us write $g_{\theta}=\sum_{j=0}^d c_{ip^j}x^{ip^{j}},$ where we assume $c_i\neq0$ and  $c_{ip^d}\neq 0$, i.e., $\mathrm{disp}(g_\theta, \infty)=d$. Then \[ \Delta(g_\theta)=c_{ip^d}x^{ip^{d+1}}-c_ix^i+\sum_{j=1}^d(c_{ip^{j-1}}-c_{ip^j})x^{ip^j},\] from which it follows that $\mathrm{disp}(\Delta(g_\theta),\infty)=d+1$, as desired.

(2).~By Lemma~\ref{lem:tree-decomposition}, $f_\tau=\Delta(g_\tau)$ for each $\tau\in\mathcal{T}_M$, and in particular for each $\tau\in\mathrm{supp}(f)$.
We consider two subcases, depending on whether $\mathrm{disp}(g,\tau)$ is finite or not.

In the first subcase, $\mathrm{disp}(g,\tau)=:d<\infty$. Let $\alpha\in \tau$ be such that $\alpha$ and $\alpha^{p^d}$ are poles of $g$. Let $\gamma\in \tau$ such that $\gamma^p=\alpha$. Then $\gamma$ is a pole of $\sigma(g)$ but not of $g$ (by maximality of $d$), whence $\gamma$ is a pole of $f$. On the other hand, $\gamma^{p^{d+1}}=\alpha^{p^d}$ is a pole of $g$ but not of $\sigma(g)$, for if $\alpha^{p^d}$ were a pole of $\sigma(g)$ then $\alpha^{p^{d+1}}$ would be a pole of $g$, again contradicting maximality of $d$. Hence $\gamma^{p^{d+1}}$ is a pole of $f$. Thus $\mathrm{disp}(f,\tau)\geq d+1$. One can show equality by contradiction: if $\alpha\in\tau$ is a pole of $f$ such that $\alpha^{p^s}$ is also a pole of $f$ for some $s> d+1$, then each of $\alpha$ and $\alpha^{p^s}$ is either a pole of $g$ or a pole of $\sigma(g)$.
This implies (after tedious but straightforward casework) that there exist $0\leq i,j\leq 1$ such that $\alpha^{p^{s+i}}$ and $\alpha^{p^j}$ are both poles of $g$, which contradicts the maximality of $d$ since in any case $s+i-j\geq s-1>d$.
Hence $\mathrm{disp}(f,\tau)=\mathrm{disp}(g,\tau)+1$ in this first subcase.

In the last remaining subcase where $\mathrm{disp}(g,\tau)=\infty$, there exists $\gamma\in\mathrm{sing}(g,\tau)\cap\mathcal{C}(\tau)$ by Lemma~\ref{lem:infinite-dispersion}. We claim $\gamma^{p^\ell}\in\mathrm{sing}(f,\tau)$ for some $\ell\in\mathbb{Z}/e\mathbb{Z}$, where $e:=e(\tau)\geq 1$ (cf.~Remark~\ref{rem:cycle-facts}, where we discussed the meaning of $\gamma^{p^\ell}$ for $\ell\in\mathbb{Z}/e\mathbb{Z}$, rather than $\ell\in\mathbb{Z}_{\geq 0}$). This will imply that $\mathrm{disp}(f,\tau)=\infty=\mathrm{disp}(g,\tau)+1$, by Lemma~\ref{lem:infinite-dispersion}.

Let us prove the claim. Note that the $\mathbb{K}$-subspace $S$ of $\mathbb{K}(x)_\tau$, consisting of rational functions none of whose poles belongs to $\mathcal{C}(\tau)$, or equivalently (by Lemma~\ref{lem:infinite-dispersion}), the $\mathbb{K}$\nobreakdash-span of the elements of $\mathbb{K}(x)_\tau$ having finite dispersion, is $\sigma$-stable\footnote{If a denominator $b\in\mathbb{K}[x]$ has no roots in $\mathcal{C}(\tau)$ then neither does $\sigma(b)$, for if $\gamma\in\mathcal{C}(\tau)$ were a root of $\sigma(b)$ then $\gamma^p\in\mathcal{C}(\tau)$  would be a root of $b$.}.
So we may assume \[g_\tau=\sum_{k=1}^m \sum_{\ell\in\,\mathbb{Z}/e\mathbb{Z}} \frac{d(k,\ell)}{(x-\gamma^{p^\ell})^k},\] where $d(k,\ell)\in\mathbb{K}$ such that $d(m,\ell)\neq 0$ for some $\ell\in\mathbb{Z}/e\mathbb{Z}$, without loss of generality, because the other possible poles of $g$ in $\tau$ cannot contribute to the possible poles of $f$ in $\mathcal{C}(\tau)$ (since $S$ is $\sigma$-stable). Then 
\begin{multline*}
\hspace{-1em}\sigma(g_\tau) =  \sum_{k=1}^m \sum_{\ell\in\,\mathbb{Z}/e\mathbb{Z}} \frac{d(k,\ell)}{\left(x^p-\gamma^{p^\ell}\right)^k} =   \sum_{i=0}^{p-1}\left(\sum_{\ell\in\,\mathbb{Z}/e\mathbb{Z}}   \frac{V^m_m\left(\zeta_p^i\gamma^{p^{\ell-1}}\right)\cdot d(m,\ell)}{\left(x-\zeta_p^i\gamma^{p^{\ell-1}}\right)^m}\right)\,+\hphantom{x} \\ \hspace{1em}+ (\text{lower-order terms}),\end{multline*} where the $V^m_m$ are as in \eqref{eq:mahler-coefficients}, and therefore\footnote{See Remark~\ref{rem:sigma-torsion}, where we systematically elaborate on the details of this computation.} \begin{multline}\label{eq:D-map-earlier} f_\tau =\Delta(g_\tau) = \sum_{\ell\in\,\mathbb{Z}/e\mathbb{Z}} \frac{V^m_m\left(\gamma^{p^{\ell}}\right)\cdot d(m,\ell+1)-d(m,\ell)}{\left(x-\gamma^{p^{\ell}}\right)^m}+\hphantom{x}\\+(\text{lower-order terms})+(\text{elements of } S).\end{multline} But the coefficients $V^m_m(\gamma^{p^\ell})\cdot d(m,\ell+1)-d(m,\ell)$ cannot be zero for every $\ell\in\mathbb{Z}/e\mathbb{Z}$, for otherwise the computation \[d(m,\ell)=d(m,\ell)\prod_{j=0}^{e-1}V^m_m\left(\gamma^{p^j}\right)=d(m,\ell)\prod_{j=0}^{e-1}\frac{\gamma^{mp^j}}{p^m\gamma^{mp^{j+1}}}=\frac{d(m,\ell)}{p^{em}},\] where the middle equality is obtained from Corollary~\ref{cor:top-coefficient}, would imply that $d(m,\ell)=0$ for every $\ell\in\mathbb{Z}/e\mathbb{Z}$. But this is impossible, concluding the proof of the claim that $f$ has a pole in $\mathcal{C}(\tau)$.
\end{proof}

\begin{cor} \label{cor:summable-dispersion} Suppose that $f\in \mathbb{K}(x)$ is Mahler summable. Then $\mathrm{disp}(f,\tau)\neq 0$ for every $\tau\in\mathrm{supp}(f)$.
\end{cor}


\section{Mahler discrete residues}\label{sec:residues}

In this section we define the Mahler discrete residues of $f\in\mathbb{K}(x)$, in increasing order of complexity: first at infinity, and then at Mahler trees $\tau\in\mathcal{T}_M$, separately in the subcase where $\tau\not\subset\mathbb{K}^\times_t$, and finally in the subcase where $\tau\subset\mathbb{K}^\times_t$ (cf.~Remark~\ref{rem:cycle-facts}).


\subsection{Mahler discrete residue at infinity}\label{sec:laurent-residues}

Here we define the Mahler discrete residue of $f\in\mathbb{K}(x)$ at $\infty$ in terms of the component $f_L\in\mathbb{K}[x,x^{-1}]$ of $f$ in \eqref{eq:f-decomposition}, and show that it forms a complete obstruction to the Mahler summability of $f_L$. The proof of Proposition~\ref{prop:laurent-residues} below follows the same strategy as that of \cite[Propositions~2.5 and 2.10]{ChenSinger2012}: we add to $f_L$ a sequence of Mahler summable elements to eventually obtain a \emph{Mahler reduction} $\bar{f}_L$ whose apparent dispersion is $0$, and then use Corollary~\ref{cor:summable-dispersion} to conclude that $f_L$ is Mahler summable if and only if this $\bar{f}_L=0$.

\begin{defn}\label{defn:laurent-residues} For $f\in\mathbb{K}(x)$, let $f_L=\sum_{j\in \mathbb{Z}}c_j x^j$ with $c_j=0$ for all but finitely many $j\in \mathbb{Z}$. The \emph{Mahler discrete residue} of $f$ at $\infty$ is the vector \[ \mathrm{dres}(f,\infty):=\left(\sum_{j\in \theta} c_j\right)_{\theta\in\,\mathbb{Z}/\mathcal{P}}\in\bigoplus_{\theta\in\,\mathbb{Z}/\mathcal{P}}\mathbb{K}.\]
\end{defn}

\begin{prop}\label{prop:laurent-residues} For $f\in\mathbb{K}(x)$, the component $f_L\in \mathbb{K}[x,x^{-1}]$ in \eqref{eq:f-decomposition} is Mahler summable if and only if $\mathrm{dres}(f,\infty)=\mathbf{0}$.
\end{prop}

\begin{proof} By Lemma~\ref{lem:trajectory-projection}, $f_L$ is Mahler summable if and only if $f_\theta$ is Mahler summable for all $\theta\in\mathbb{Z}/\mathcal{P}$. We shall show $f_\theta$ is Mahler summable if and only if the component of $\mathrm{dres}(f,\infty)_\theta=0$. We prove this separately in two cases: (1).~$\theta=\{0\}$; and (2).~$\theta\neq \{0\}$.

(1).~For $\theta=\{0\}$, $f_{\{0\}}=c_0=\mathrm{dres}(f,\infty)_{\{0\}}$ by Definition~\ref{defn:laurent-residues}. If $\mathrm{dres}(f,\infty)_{\{0\}}=0$, then $f_{\{0\}}=0$ is Mahler summable. On the other hand, if $\mathrm{dres}(f,\infty)_{\{0\}}\neq 0$ then $f_{\{0\}}\neq 0$ and $\mathrm{disp}(f_{\{0\}},\infty)=0$ by Definition~\ref{defn:dispersion}, so $f_{\{0\}}$ is not Mahler summable by Corollary~\ref{cor:summable-dispersion}. 

(2).~ Suppose $\theta\neq\{0\}$. The claim is trivial in case $f_\theta=0$; assume $f_\theta\neq 0$. Let us write $ f_\theta=\sum_{j\in\theta} c_jx^j\in\mathbb{K}[x,x^{-1}]_\theta$, where: $c_j=0$ for all but finitely many $j\in \theta$ and $c_j\neq 0$ for at least one $j\in \theta$. Let us write $\theta = \{ip^n \ | \ n\in\mathbb{Z}_{\geq 0}\}$ for $i\in\theta$ such that $p\nmid i$. Let $h\in\mathbb{Z}_{\geq 0}$ be maximal such that $c_{ip^{h}}\neq 0$. Let us define recursively: $g_{\theta}^{(0)}:=0$; and, if $h\geq 1$, then set \[  g_{\theta}^{(n+1)}:=\sum_{k=0}^{n}\left(\sum_{\ell=0}^k c_{ip^\ell}\right)x^{ip^k}=g_{\theta}^{(n)}+\left(\sum_{\ell=0}^{n}c_{ip^\ell}\right)x^{ip^{n}}\]
for $0\leq n\leq h-1$. A straightforward induction argument shows:
\begin{equation}  \bar{f}_{\theta}^{(n)}:=f_{\theta}+\Delta\left(g_{\theta}^{(n)}\right)=\sum_{k=n+1}^{h+1}c_{ip^k}x^{ip^k}+ \left(\sum_{\ell=0}^n c_{ip^\ell}\right)x^{ip^{n}}\label{eq:harmless}\end{equation} 
for each $0\leq n\leq h$, whence $\bar{f}_{\theta}^{(h)} =(\mathrm{dres}(f,\infty)_{\theta})\cdot x^{ip^{h}}  $. The harmless summand for $k=h+1$ in \eqref{eq:harmless} is included so that the sum makes sense for $n=h$, but $c_{ip^k}=0$ for every $k>h$. We see that $\bar{f}_\theta^{(h)}$ is Mahler summable if and only if $f_\theta$ is Mahler summable. In particular, if $\bar{f}_\theta^{(h)}=0$ then $f_\theta$ is Mahler summable. But if $\bar{f}_\theta^{(h)}\neq 0$ then $\smash{\mathrm{disp}(\bar{f}_{\theta}^{(h)},\infty)=0}$, and by Corollary~\ref{cor:summable-dispersion} $\bar{f}_\theta^{(h)}$ is not Mahler summable, so neither is $f_\theta$. Clearly, $\bar{f}_\theta^{(h)}=0\Leftrightarrow\mathrm{dres}(f,\infty)_\theta=0$.\end{proof}

\begin{rem} \label{rem:laurent-remainder}For $\{0\}\neq \theta\in\mathbb{Z}/\mathcal{P}$ such that $f_\theta\neq 0$, the elements $\bar{f}_{\theta}^{(h)}, g_{\theta}^{(h)}\in\mathbb{K}[x,x^{-1}]_\theta$ constructed in the proof of Proposition~\ref{prop:laurent-residues} are the $\theta$-components of the $\bar{f},g\in \mathbb{K}(x)$ in \eqref{eq:mahler-reduction}. If $f_\theta=0$, then we define $\bar{f}_\theta:=0=:g_\theta$. In any case, we set $\bar{f}_{\{0\}}:=f_{\{0\}}$ and $g_{\{0\}}:=0$.
\end{rem}


\subsection{Mahler discrete residues at Mahler trees} \label{sec:tree-residues}

Here we define the Mahler discrete residues of $f\in\mathbb{K}(x)$ at a Mahler tree $\tau\in\mathcal{T}_M$, in terms of the partial fraction decomposition of the component $f_\tau\in\mathbb{K}(x)_\tau$ in Definition~\ref{defn:tree-projection}, and show they comprise a complete obstruction to the Mahler summability of $f_\tau$. We proceed separately in the non-torsion case $\tau\not\subset\mathbb{K}^\times_t$ and the torsion case $\tau\subset\mathbb{K}^\times_t$ (cf.~Remark~\ref{rem:cycle-facts}), depending on which case we represent the poles of $f_\tau$ in a particular manner.

\begin{lem}\label{lem:height} For $f\in\mathbb{K}(x)$ and $\tau\in\mathrm{supp}(f)\cap\mathcal{T}_M$ such that $\tau\not\subset\mathbb{K}^\times_t$, there exists $\gamma\in\mathrm{sing}(f,\tau)$ and $h\in\mathbb{Z}_{\geq 0}$ such that \[ \mathrm{sing}(f,\tau)\subseteq\beta_h(\gamma):=\bigl\{\zeta_{p^{n}}^i\gamma^{p^{h-n}} \ \big| \ 0\leq n \leq h; \ i\in\mathbb{Z}/p^n\mathbb{Z}\bigr\}.\] Moreover, the elements $\zeta_{p^{n}}^i\gamma^{p^{h-n}}\in\beta_h(\gamma)$ are uniquely determined by $0\leq n \leq h$ and $i\in\mathbb{Z}/p^{n}\mathbb{Z}$, relative to the choice of $\gamma\in\mathrm{sing}(f,\tau)$.
\end{lem}

\begin{proof}Note that the set $\beta_h(\gamma)$ (mnemonic: ``bouquet'' of height $h$ at $\gamma$) is precisely the union of the sets of roots of the $y$-polynomials $y^{p^{n}}-\gamma^{p^h}=0$ for all $0\leq n \leq h$. The elements of $\beta_h(\gamma)$ are uniquely determined by $n$ and $i$ (relative to the choice of $\gamma$), because if we had $\smash{\zeta_{p^{m}}^j\gamma^{p^{h-m}}  =\zeta_{p^n}^i\gamma^{p^{h-n}}}$, then this would force $m=n$, for otherwise $\gamma\in\mathbb{K}^\times_t$ contradicting our assumptions, and then $\zeta_{p^{n}}^j=\zeta_{p^{n}}^i$ implies that $j=i$. Let us now show that for any finite set $S\subset \tau$ there exist $\gamma\in S$ and $h\in\mathbb{Z}_{\geq 0}$ such that $S\subseteq \beta_h(\gamma)$. For $\alpha\in S$, let $h(\alpha)\in\mathbb{Z}_{\geq 0}$ be minimal such that $\alpha^{p^{h(\alpha)}}\in\xi^\mathcal{P}$ for every $\xi\in S$, where $\xi^\mathcal{P}:=\{\xi^{p^t} \ | \ t\in\mathbb{Z}_{\geq 0}\}$. Choose $\gamma\in S$ such that $h(\gamma)=:h$ is maximal among all elements of $S$. We claim that $\alpha^{p^{h(\alpha)}}=\gamma^{p^h}$ for every $\alpha\in S$, which will conclude the proof, since $h(\alpha)\leq h$ for every $\alpha\in S$. To prove the claim, note that in any case there exist $t,r\in\mathbb{Z}_{\geq 0}$ such that $\alpha^{p^t}=\gamma^{p^h}$ and $\alpha^{p^{h(\alpha)}}=\gamma^{p^r}$, and the minimality of $h(\alpha)$ and $h$ then imply $t\geq h(\alpha)$ and $r\geq h$. But then \[\gamma^{p^h}=\alpha^{p^t}=\bigl(\alpha^{p^{h(\alpha)}}\bigr)^{p^{t-h(\alpha)}} = \bigl(\gamma^{p^r}\bigr)^{p^{t-h(\alpha)}}=\gamma^{p^{r+t-h(\alpha)}},\] and since $\gamma\notin\mathbb{K}^\times_t$ we obtain that $h+h(\alpha)=r+t$, from which it follows that $r=h$ and $t=h(\alpha)$, as claimed.
\end{proof}

\begin{lem}\label{lem:root-representation} Let $\tau\in\mathcal{T}_M$ with $\tau\subset\mathbb{K}^\times_t$ and $e:=e(\tau)$. Choose $\gamma\in\mathcal{C}(\tau)$. Then for $\alpha\in\tau$ there are unique $n\in\mathbb{Z}_{\geq 0}$, $i\in\mathbb{Z}/p^n\mathbb{Z}$, and $\ell\in\mathbb{Z}/e\mathbb{Z}$, with either $n=i=0$ or $p\nmid i$, such that \[\alpha=\zeta_{p^n}^i\gamma^{p^{\ell-\pi_e(n)}}.\]\end{lem}

\begin{proof} There exist integers $n,t\in\mathbb{Z}_{\geq 0}$ such that $\alpha^{p^n}=\gamma^{p^t}$, and we may take this $n$ to be as small as possible and replace $t$ with $\ell:=\pi_e(t)$. The $p^n$ distinct solutions to $y^{p^n}=\gamma^{p^\ell}$, one of which is $y=\alpha$, are all of the form $\zeta_{p^n}^i\gamma^{p^{\ell-\pi_e(n)}}$ for $i\in\mathbb{Z}/p^n\mathbb{Z}$. Since $n$ is minimal, either $n=i=0$ or else $1\leq i\leq p^n-1$ with $p\nmid i$. \end{proof}

\begin{defn}\label{defn:height} For $f\in\mathbb{K}(x)$ and $\tau\in\mathrm{supp}(f)\cap\mathcal{T}_M$, the \emph{height} of $f$ at $\tau$, denoted by $h(f,\tau)$, is defined as follows.
\begin{itemize}
\item If $\tau\not\subset \mathbb{K}^\times_t$, $h(f,\tau)$ is the smallest $h\in\mathbb{Z}_{\geq 0}$ such that $\mathrm{sing}(f,\tau)$ is contained in $\beta_h(\gamma)$ for some $\gamma\in\mathrm{sing}(f,\tau)$ as in Lemma~\ref{lem:height}.

\item If $\tau\subset\mathbb{K}^\times_t$, $h(f,\tau)$ is the smallest $h\in\mathbb{Z}_{\geq 0}$ such that $\alpha^{p^h}$ belongs to $\mathcal{C}(\tau)$ for every $\alpha\in\mathrm{sing}(f,\tau)$.
\end{itemize}
\end{defn}

\begin{rem}\label{rem:height-dispersion}
Note that we always have $h(f,\tau)<\infty$. One can show that $h(f,\tau)\geq \mathrm{disp}(f,\tau)$ provided that $\mathrm{disp}(f,\tau)<\infty$, but even in this case the inequality may be strict.\end{rem}


\subsubsection{Mahler discrete residues at Mahler trees: the non-torsion case}\label{sec:non-torsion-residues}

\begin{lem}\label{lem:non-torsion-expansion} Let $f\in\mathbb{K}(x)$ and suppose $\tau\in\mathrm{supp}(f)\cap\mathcal{T}_M$ such that $\tau\not\subset\mathbb{K}^\times_t$. Then there exists $\gamma\in\mathrm{sing}(f,\tau)$ such that the partial fraction decomposition of $f_\tau$ is of the form \begin{equation}\label{eq:non-torsion-expansion}f_\tau=\sum_{k=1}^m\sum_{n=0}^{h}\left(\sum_{i\in \, \mathbb{Z}/p^{n}\mathbb{Z}} \frac{c_\gamma(k,n,i)}{\left(x-\zeta_{p^{n}}^i\gamma^{p^{h-n}}\right)^k}\right),\end{equation} where $m\geq 1$ is the highest order of a pole of $f$ in $\mathrm{sing}(f,\tau)$ and the height $h:=h(f,\tau)$ is as in Definition~\ref{defn:height}.

The coefficients $c_\gamma(k,n,i)\in\mathbb{K}$ are uniquely determined by $f$ and the choice of $\gamma$, and moreover for any $\gamma,\tilde{\gamma} \in \tau$ as above we have $\tilde{\gamma}=\zeta_{p^{h}}^j\gamma$ for some $j\in\mathbb{Z}/p^{h}\mathbb{Z}$, and 
\begin{equation}\label{eq:non-torsion-coefficients}  c_{\tilde{\gamma}}(k,n,i)=c_\gamma\left(k,n,i+\pi^{h}_{n}(j)\right).\end{equation}\end{lem}

\begin{proof} We obtain the existence of $\gamma\in\mathrm{sing}(f,\tau)$ such that $\mathrm{sing}(f,\tau)\subseteq \beta_h(\gamma)$ by Lemma~\ref{lem:height} and Definition~\ref{defn:height}. The existence and uniqueness of the coefficients $c_\gamma(k,n,i)\in\mathbb{K}$ satisfying \eqref{eq:non-torsion-expansion} follows directly from the existence and uniqueness of partial fraction decompositions, since in this case the elements $\zeta_{p^n}^i\gamma^{p^{h-n}}   \in \beta_h(\gamma)$ are uniquely determined by $n$ and $i$ (relative to the choice of $\gamma$), by Lemma~\ref{lem:height}. For any $\tilde{\gamma}\in\mathrm{sing}(f,\tau)$ such that $\mathrm{sing}(f,\tau) \subseteq \beta_h(\tilde{\gamma})$ we would have that $\tilde{\gamma}^{p^{\tilde{n}}}=\gamma^{p^h}$ and $\gamma^{p^n}=\tilde{\gamma}^{p^h}$ such that $0\leq n,\tilde{n}\leq h$, which forces $\tilde{n}=h=n$ since $\gamma\notin\mathbb{K}^\times_t$. Hence $\tilde{\gamma}=\zeta_{p^h}^j\gamma$ for some $j\in\mathbb{Z}/p^h\mathbb{Z}$, and the computation \[ \frac{c_{\tilde{\gamma}}(k,n,i)}{\left(x-\zeta_{p^n}^i\tilde{\gamma}^{p^{h-n}}\right)^k}=\frac{c_{\tilde{\gamma}}(k,n,i)}{\left(x-\zeta_{p^n}^{i+j}\gamma^{p^{h-n}}\right)^k}=\frac{c_{\gamma}(k,n,i+j)}{\left(x-\zeta_{p^n}^{i+j}\gamma^{p^{h-n}}\right)^k}\] implies the transformation formula \eqref{eq:non-torsion-coefficients}.\end{proof}

\begin{rem}\label{rem:sigma-non-torsion} Writing $f_\tau$ as in \eqref{eq:non-torsion-expansion}, let us compute the effect of $\sigma$:
\begin{align}\label{eq:sigma-non-torsion}
{}& \sigma\left(\sum_{k=1}^m\left(\sum_{i\in\,\mathbb{Z}/p^{n}\mathbb{Z}} c_\gamma(k,n,i)\cdot\left(x-\zeta_{p^{n}}^i\gamma^{p^{h-n}}\right)^{-k}\right)\right)  \\
{}&=\sum_{k=1}^m\left(\sum_{i\in\,\mathbb{Z}/p^{n+1}\mathbb{Z}} \frac{\sum_{s=k}^mV^s_k\left(\zeta_{p^{n+1}}^i\gamma^{p^{h-(n+1)}}\right)\cdot c_\gamma\left(s,n,\pi^{n+1}_n(i)\right)}{\left(x-\zeta_{p^{n+1}}^i\gamma^{p^{h-(n+1)}}\right)^k}\right)\notag
\end{align}
for each $0\leq n \leq h-1$, where the $V^s_k$ are as in \eqref{eq:mahler-coefficients} for $k\leq s\leq m$.
\end{rem}

\begin{defn}\label{defn:non-torsion-residues} For $f\in\mathbb{K}(x)$ and $\tau\in\mathcal{T}_M$ with $\tau\not\subset\mathbb{K}^\times_t$, the \emph{Mahler discrete residue} of $f$ at $\tau$ of degree $k\in\mathbb{N}$ is the vector $\mathrm{dres}(f,\tau,k)\in\bigoplus_{\alpha\in\tau}\mathbb{K}$ defined in terms of the $c_\gamma(k,n,i)$ in the partial fraction decomposition of $f_\tau$ in Lemma~\ref{lem:non-torsion-expansion} as follows.

Set $\mathrm{dres}(f,\tau,k):=\mathbf{0}$ if $\tau\notin\mathrm{supp}(f)$ or if $k>m$. For $\tau\in\mathrm{supp}(f)$ and $\alpha\in \tau$, the component $\mathrm{dres}(f,\tau,k)_\alpha:=0$ whenever $\alpha^{p^h}\neq\gamma^{p^h}  $.

For $1\leq k\leq m$ and $\alpha=\zeta_{p^h}^i\gamma$ with $i\in\mathbb{Z}/p^h\mathbb{Z}$, the component  \[\mathrm{dres}(f,\tau,k)_{ \alpha}:=\hat{c}_\gamma(k,h,i);\] where for $0\leq n\leq h$ and $i\in\mathbb{Z}/p^n\mathbb{Z}$ we define recursively (in $n$): 
\begin{align}\label{eq:non-torsion-residue-coefficients}
\hat{c}_\gamma(k,0,0) &:=c_\gamma(k,0,0); \ \ \text{and, if} \  h\geq 1, \ \text{then set}  
\\ \hat{c}_\gamma(k,n,i) &:=c_\gamma(k,n,i)+\sum_{s=k}^m V^s_k\left(\zeta_{p^n}^i\gamma^{p^{h-n}}\right)\cdot \hat{c}_\gamma\left(s,n-1,\pi^n_{n-1}(i)\right)\notag
\end{align} 
(cf.~\eqref{eq:sigma-non-torsion}) for $1\leq n \leq h$, where the $V^s_k$ are as in \eqref{eq:mahler-coefficients} for $k\leq s\leq m$.
\end{defn}

\begin{rem} \label{rem:non-torsion-residue-independent} Note that the definition of $\mathrm{dres}(f,\tau,k)$ for $\tau\not\subset\mathbb{K}^\times_t$ given above is independent of the choice of $\gamma\in\mathrm{sing}(f,\tau)$, because for any possibly different $\tilde{\gamma}=\zeta_{p^h}^j\gamma$ with $j\in\mathbb{Z}/p^h\mathbb{Z}$ we obtain $\smash{\zeta_{p^h}^i\tilde{\gamma}=\zeta_{p^h}^{i+j}\gamma=:\alpha}$. The equality of the expressions \[\hat{c}_{\tilde{\gamma}}(k,h,i)=\mathrm{dres}(f,\tau,k)_\alpha=\hat{c}_\gamma(k,h,i+j)\] follows from \eqref{eq:non-torsion-coefficients}, since $\zeta_{p^n}^i\tilde{\gamma}^{p^{h-n}}=\zeta_{p^n}^{i+j}\gamma^{p^{h-n}}$ for all $i\in\mathbb{Z}/p^n\mathbb{Z}$, and therefore $\hat{c}_{\tilde{\gamma}}(k,n,i)=\hat{c}_\gamma(k,n,i+\pi^h_n(j))$ for every $0\leq n\leq h$.
\end{rem}


\subsubsection{Mahler discrete residues at Mahler trees: the torsion case} \label{sec:torsion-residues}

\begin{lem}\label{lem:torsion-expansion} Let $f\in \mathbb{K}(x)$ and suppose $\tau\in\mathrm{supp}(f)$ such that $\tau\subset\mathbb{K}^\times_t$. Then for any $\gamma\in\mathcal{C}(\tau)$ the partial fraction decomposition of $f_\tau$ is of the form
\begin{equation}\label{eq:torsion-expansion}  f_{\tau}=\sum_{k=1}^m\sum_{n=0}^h\left(\sideset{}{'}\sum_{i\in\mathbb{Z}/p^n\mathbb{Z}}\left(\sum_{\ell\in\mathbb{Z}/e\mathbb{Z}}\frac{d_\gamma(k,n,i,\ell)}{\left(x-\zeta_{p^n}^i\gamma^{p^{\ell-\pi_e(n)}}\right)^k}\right)\right),
\end{equation}
where: $m\geq 1$ is the highest order of a pole of $f$ in $\mathrm{sing}(f,\tau)$; the height $h:=h(f,\tau)$ is as in Definition~\ref{defn:height}; the restricted sum $\sum'$ is taken over $i\in\mathbb{Z}/p^n\mathbb{Z}$ such that $p\nmid i$ whenever $n\neq 0$; and $e:=e(\tau)\geq 1$.

The coefficients $d_\gamma(k,n,i,\ell)\in\mathbb{K}$ are uniquely determined by $f$ and $\gamma$, and moreover for any $\gamma,\tilde{\gamma}\in\mathcal{C}([\alpha]_M)$ we have $\smash{\tilde{\gamma}=\gamma^{p^j}}$ for some $j\in\mathbb{Z}/e\mathbb{Z}$, and \begin{equation}\label{eq:coefficients-torsion-expansion} d_{\tilde{\gamma}}(k,n,i,\ell)=d_\gamma(k,n,i,\ell+j).\end{equation}\end{lem}

\begin{proof} If $\tau\subset\mathbb{K}^\times_t$ then $e\geq 1$ (cf.~Remark~\ref{rem:cycle-facts}). We then have by Lemma~\ref{lem:root-representation} that for any given choice of $\gamma\in\mathcal{C}(\tau)$ the elements $\alpha\in\tau$ can be written uniquely as $\alpha=\zeta_{p^n}^i\gamma^{p^{\ell-\pi_e(n)}}$ for some $n\in\mathbb{Z}_{\geq 0}$ and $i\in\mathbb{Z}/p^n\mathbb{Z}$ such that either $n=i=0$ or else $p\nmid i$. It follows that the apparent poles in \eqref{eq:torsion-expansion} are all distinct, and therefore the coefficients $d_\gamma(k,n,i,\ell)\in\mathbb{K}$ are uniquely determined by $f$ and $\gamma$. The set of elements $\alpha\in\tau$ such that $\alpha^{p^h}\in\mathcal{C}(\tau)$ are precisely the $\smash{\alpha=\zeta_{p^n}^i\gamma^{p^{\ell-\pi_e(n)}}}$ with $n\leq h$. It also follows from Lemma~\ref{lem:root-representation} that for any other $\tilde{\gamma}\in\mathcal{C}(\tau)$ there exists $j\in\mathbb{Z}/e\mathbb{Z}$ such that $\tilde{\gamma}=\gamma^{p^j}$, and therefore the computation \[ \frac{d_{\tilde{\gamma}}(k,n,i,\ell)}{\bigl(x-\zeta_{p^n}^i\tilde{\gamma}^{p^{\ell-\pi_e(n)}}\bigr)^k} = \frac{d_{\tilde{\gamma}}(k,n,i,\ell)}{\bigl(x-\zeta_{p^n}^i\gamma^{p^{\ell+j-\pi_e(n)}}\bigr)^k} = \frac{d_{\gamma}(k,n,i,\ell+j)}{\bigl(x-\zeta_{p^n}^i\gamma^{p^{\ell+j-\pi_e(n)}}\bigr)^k}\] implies the transformation formula \eqref{eq:coefficients-torsion-expansion}.
\end{proof}

\begin{rem} \label{rem} \label{rem:sigma-torsion} Writing $f_\tau$ as in \eqref{eq:torsion-expansion}, let us compute the effect of $\sigma$:
\begin{align}\label{eq:sigma-torsion0}
\sigma&   \left(\sum_{k=1}^m\sum_{\ell\in\,\mathbb{Z}/e\mathbb{Z}}d_\gamma(k,0,0,\ell)\cdot\left(x-\gamma^{p^{\ell}}\right)^{-k}\right)  \\
{}&=\sum_{k=1}^m\left(\sum_{\ell\in\,\mathbb{Z}/e\mathbb{Z}} \frac{\sum_{s=k}^mV^s_k\left(\gamma^{p^\ell}\right)\cdot d_\gamma\left(s,0,0,\ell+1\right)}{\left(x-\gamma^{p^\ell}\right)^k}\right)+\hphantom{x} \notag\\
{}&\hphantom{xxx}+\sum_{k=1}^m\sum_{i=1}^{p-1}\left(\sum_{\ell\in\,\mathbb{Z}/e\mathbb{Z}}\frac{\sum_{s=k}^mV^s_k\left(\zeta_p^i\gamma^{p^{\ell-1}}\right)\cdot d_\gamma\left(s,0,0,\ell\right)}{\left(x-\zeta_p^i\gamma^{p^{\ell-1}}\right)^k}\right);\notag
\end{align}
and for $n\geq 1$ and each $\ell\in\mathbb{Z}/e\mathbb{Z}$ we have 
\begin{align}\label{eq:sigma-torsion+}
{}&  \sigma\left(\sum_{k=1}^m\sum_{i\in\mathbb{Z}/p^n\mathbb{Z}}'d_\gamma(k,n,i,\ell)\cdot\left(x-\zeta_{p^n}^i\gamma^{p^{\ell-\pi_e(n)}}\right)^{-k}\right) \\
{}&=\sum_{k=1}^m\left(\sideset{}{'}\sum_{i\in\,\mathbb{Z}/p^{n+1}\mathbb{Z}} \frac{\sum_{s=k}^mV^s_k\left(\zeta_{p^{n+1}}^i\gamma^{p^{\ell -\pi_e(n+1)}}\right)\cdot d_\gamma\left(s,n,\pi^{n+1}_n(i),\ell\right)}{\left(x-\zeta_{p^{n+1}}^i\gamma^{p^{\ell-\pi_e(n+1)}}\right)^k}\right);\notag
\end{align}
where the $V^s_k$ are as in \eqref{eq:mahler-coefficients} for $k\leq s\leq m$.
\end{rem}

The following technical lemma is essential for the definition of Mahler discrete residues in the torsion case. The map $\mathcal{D}^{(m)}_\gamma$ defined below already appeared (anonymously) in \eqref{eq:D-map-earlier}. It captures the effect of $\Delta$ on the (classical/continuous) residues at poles in the Mahler cycle $\mathcal{C}(\tau)$ (see Definition~\ref{defn:cycles}), according to the computation \eqref{eq:sigma-torsion0}. 

\begin{lem}\label{lem:cyclic-coefficients}Let $\tau\in\mathcal{T}_M$ with $\tau\subset\mathbb{K}^\times_t$ and $e:=e(\tau)$. For $\gamma\in\mathcal{C}(\tau)$ and $m\in\mathbb{N}$, define $\mathcal{D}_\gamma^{(m)}:\mathbb{K}^{m\times e} \rightarrow \mathbb{K}^{m\times e}$ by \begin{equation}\label{eq:cyclic-coefficients}\mathcal{D}^{(m)}_\gamma:(c_{k,\ell})_{\substack{1\leq k \leq m \\ \ell\in\,\mathbb{Z}/e\mathbb{Z}}}\mapsto\left(d_{k,\ell}\right)_{\substack{1\leq k \leq m \\ \ell\in\,\mathbb{Z}/e\mathbb{Z}}},\end{equation} where $d_{k,\ell}:=c_{k,\ell}-\sum_{s=k}^mV^s_k\bigl(\gamma^{p^\ell}\bigr)\cdot c_{s,\ell+1}$, and where the $V^s_k$ are as in \eqref{eq:mahler-coefficients}. Then $\mathcal{D}_\gamma^{(m)}$ is invertible and has no non-trivial fixed points.\end{lem}

\begin{proof} Let $\mathbf{0}\neq (c_{k,\ell})\in\mathbb{K}^{m\times e}$, and write $(d_{k,\ell}):=\mathcal{D}^{(m)}_\gamma(c_{k,\ell})$. Let $1\leq r \leq m$ be as large as possible such that $c_{r,\ell}\neq 0$ for some $\ell\in\mathbb{Z}/e\mathbb{Z}$. To see that $(d_{k,\ell}) \neq \mathbf{0}$, note that, for each $\ell\in\mathbb{Z}/e\mathbb{Z}$, \begin{equation}\label{eq:eigenspace}\notag   d_{r,\ell}= c_{r,\ell}-\sum_{s=r}^mV^s_r\bigl(\gamma^{p^\ell}\bigr)\cdot c_{s,\ell+1}= c_{r,\ell}-V^r_r\bigl(\gamma^{p^\ell}\bigr)\cdot c_{r,\ell+1}\end{equation} because $c_{s,\ell+1}=0$ whenever $s>r$, and we see just as at the end of proof of Proposition~\ref{prop:summable-dispersion} that the $d_{r,\ell}$ cannot be zero for all $\ell\in\mathbb{Z}/e\mathbb{Z}$ because this would imply that every $c_{r,\ell}=0$, contradicting our choice of $r$. Moreover, we also cannot have $d_{k,\ell}=c_{k,\ell}$ for every $1\leq k \leq m$ and $\ell\in\mathbb{Z}/e\mathbb{Z}$, for this would also imply that $c_{r,\ell}=0$ for every $\ell\in\mathbb{Z}/e\mathbb{Z}$, again contradicting our choice of $r$. \end{proof}

\begin{defn}\label{defn:cyclic-coefficients} With notation as in Lemma~\ref{lem:cyclic-coefficients}, the inverse of $\mathcal{D}_\gamma^{(m)}$ is denoted by $\mathcal{L}_\gamma^{(m)}$.
\end{defn}

\begin{defn}\label{defn:torsion-residues} For $f\in\mathbb{K}(x)$ and $\tau\in\mathcal{T}_M$ with $\tau\subset\mathbb{K}^\times_t$, the \emph{Mahler discrete residue} of $f$ at $\tau$ of degree $k\in\mathbb{N}$ is the vector $\mathrm{dres}(f,\tau,k)\in\bigoplus_{\alpha\in\tau}\mathbb{K}$
defined in terms of the $d_\gamma(k,n,i,\ell)$ in the partial fraction decomposition of $f_\tau$ in Lemma~\ref{lem:torsion-expansion} as follows.

Set $\mathrm{dres}(f,\tau,k):=\mathbf{0}$ if $\tau\notin\mathrm{supp}(f)$ or if $k>m$. For $\tau\in\mathrm{supp}(f)$ and $\alpha\in\tau$, the component $\mathrm{dres}(f,\tau,k)_\alpha:=0$ whenever the smallest integer $r\in\mathbb{Z}_{\geq 0}$ such that $\alpha^{p^r}\in\mathcal{C}(\tau)$ is different from $h$.

If $h=0$, then for $1\leq k \leq m$ and $\alpha=\gamma^{p^\ell}\in\mathcal{C}(\tau)$ with $\ell\in\mathbb{Z}/e\mathbb{Z}$, the component \[\mathrm{dres}(f,\tau,k)_{\gamma^{p^\ell}}:=d_\gamma(k,0,0,\ell).\]

\noindent If $h\geq 1$, then for $1\leq k \leq m$ and $\alpha=\zeta_{p^h}^i\gamma^{p^{\ell-\pi_e(h)}}$ with $i\in\mathbb{Z}/p^h\mathbb{Z}$ such that $p\nmid i$ and $\ell\in\mathbb{Z}/e\mathbb{Z}$, the component \[\mathrm{dres}(f,\tau,k)_{\alpha}:=\hat{d}_\gamma(k,h,i,\ell);\quad \text{where we set}\]
\begin{gather}\label{eq:torsion-residue-coefficient0}\hat{d}_\gamma(k,0,0,\ell) :=c_\gamma(k,\ell), \qquad \text{with}  \\ 
\bigl(c_\gamma(k,\ell)\bigr)_{\substack{1\leq k \leq m \\ \ell\in\,\mathbb{Z}/e\mathbb{Z}}}:=\mathcal{L}^{(m)}_\gamma\biggl(\bigl(d_\gamma(k,0,0,\ell)\bigr)_{\substack{1\leq k \leq m \\ \ell\in\,\mathbb{Z}/e\mathbb{Z}}}\biggr)\label{eq:torsion-residue-coefficient-L}\end{gather} for the linear map $\mathcal{L}^{(m)}_\gamma$ in Definition~\ref{defn:cyclic-coefficients}; and for $1\leq n \leq h$ and $i\in\mathbb{Z}/p^n\mathbb{Z}$ with $p\nmid i$ we define recursively (in $n$):
\begin{align}\label{eq:torsion-residue-coefficients}\hat{d}_\gamma(k,n,i,\ell):={}&
d_\gamma(k,n,i,\ell)+\hphantom{x} \\ {}&+  \sum_{s=k}^m V^s_k\bigl(\zeta_{p^n}^i\gamma^{p^{\ell-\pi_e(n)}}\bigr)\cdot \hat{d}_\gamma\bigl(s,n-1,\pi^n_{n-1}(i),\ell\bigr),\notag\end{align} (cf.~\eqref{eq:sigma-torsion+}) where the $V^s_k$ are as in \eqref{eq:mahler-coefficients}.\end{defn}

\begin{rem} \label{rem:torsion-residue-independent} Note that the definition of $\mathrm{dres}(f,\tau,k)$ for $\tau\subset\mathbb{K}^\times_t$ given above is independent of the choice of $\gamma\in\mathcal{C}(\tau)$, because for any possibly different $\tilde{\gamma}=\gamma^{p^j}$ with $j\in\mathbb{Z}/e\mathbb{Z}$ we obtain $\zeta_{p^h}^i\tilde{\gamma}^{p^{\ell-\pi_e(h)}}=\zeta_{p^h}^i\gamma^{p^{\ell+j-\pi_e(h)}}=:\alpha$. The equality of the expressions \[\hat{d}_{\tilde{\gamma}}(k,h,i,\ell)=\mathrm{dres}(f,\tau,k)_\alpha=\hat{d}_\gamma(k,h,i,\ell+j)\] follows from \eqref{eq:coefficients-torsion-expansion}, after observing that $\mathcal{D}^{(m)}_{\tilde{\gamma}}\circ\mathrm{cyc}_j=\mathcal{D}^{(m)}_\gamma$, where $\mathrm{cyc}_j:\mathbb{K}^{m\times e}\rightarrow\mathbb{K}^{m\times e}:(c_{k,\ell})\mapsto (c_{k,\ell+j})$ for $j\in\mathbb{Z}/e\mathbb{Z}$. It follows that $\mathrm{cyc}_j\circ\mathcal{L}^{(m)}_{\tilde{\gamma}}=\mathcal{L}_\gamma^{(m)}$ and therefore $\hat{d}_{\tilde{\gamma}}(k,n,i,\ell)=\hat{d}_\gamma(k,n,i,\ell+j)$ for every $0\leq n\leq h$.
\end{rem}


\subsection{Proof of the Main Theorem}\label{sec:proof}

Our proof of Proposition~\ref{prop:tree-residues} below follows a strategy similar to that of \cite[Propositions~2.5 and 2.10]{ChenSinger2012}: we add to $f_\tau$ a sequence of Mahler summable elements to eventually obtain a \emph{Mahler reduction} $\bar{f}_\tau$ whose apparent dispersion is $0$, and then use Corollary~\ref{cor:summable-dispersion} to conclude that $f_\tau$ is Mahler summable if and only if this $\bar{f}_\tau=0$.

There is a wrinkle: in case $\tau\in\mathrm{supp}(f)\cap\mathbb{K}^\times_t$ and the height $h(f,\tau)=0$ (see Definition~\ref{defn:height}), $\mathrm{disp}(f,\tau)=\infty$ by Lemma~\ref{lem:infinite-dispersion}. Corollary~\ref{cor:summable-dispersion} remains silent in this case, for which we provide a specialized argument that relies on the technical Lemma~\ref{lem:cyclic-coefficients}.

\begin{prop} \label{prop:tree-residues} For $f\in\mathbb{K}(x)$ and $\tau\in\mathcal{T}_M$, the component $f_\tau$ is Mahler summable if and only if $\mathrm{dres}(f,\tau,k)=\mathbf{0}$ for every $k\in\mathbb{N}$.
\end{prop}

\begin{proof} The Proposition is trivial for $\tau\notin\mathrm{supp}(f)\Leftrightarrow f_\tau=0$. Assume from now on that $\tau\in\mathrm{supp}(f)$. The proofs in the different cases $\tau\not\subset\mathbb{K}^\times_t$ versus $\tau\subset\mathbb{K}^\times_t$ proceed in parallel below.

Write $f_\tau$ as in Lemma~\ref{lem:non-torsion-expansion} if $\tau\not\subset\mathbb{K}^\times_t$ and as in Lemma~\ref{lem:torsion-expansion} if $\tau\subset\mathbb{K}^\times_t$. Let us define recursively: $g_{\tau}^{(0)}:=0$; and, if $h:=h(f,\tau)\geq 1$ as in Definition~\ref{defn:height}, then for $0\leq n \leq h-1$ set
\begin{align*}
  g_{\tau}^{(n+1)} &:=g_{\tau}^{(n)}+\sum_{k=1}^m\left(\sum_{i\in\,\mathbb{Z}/p^{n}\mathbb{Z}}\frac{\hat{c}_\gamma(k,n,i)}{\left(x-\zeta_{p^{n}}^i\gamma^{p^{h-n}}\right)^k}\right)
\intertext{in case $\tau\not\subset\mathbb{K}^\times_t$, with $\hat{c}_\gamma(k,n,i)$ as in \eqref{eq:non-torsion-residue-coefficients}; and}
g_{\tau}^{(n+1)} &:=g_{\tau}^{(n)}+\sum_{k=1}^m\left(\,\sideset{}{'}\sum_{i\in\mathbb{Z}/p^n\mathbb{Z}}\left(\sum_{\ell\in\mathbb{Z}/e\mathbb{Z}}\frac{\hat{d}_\gamma(k,n,i,\ell)}{\left(x-\zeta_{p^n}^i\gamma^{p^{\ell-\pi_e(n)}}\right)^k}\right)\right)
\end{align*}
in case $\tau\subset\mathbb{K}^\times_t$, with $\hat{d}_\gamma(k,n,i,\ell)$ as in \eqref{eq:torsion-residue-coefficient0} for $n=0$ and as in \eqref{eq:torsion-residue-coefficients} for $n \geq 1$. Setting $\bar{f}_{\tau}^{(n)}:=f_{\tau}+\Delta\left( g_{\tau}^{(n)}\right)$, an induction argument then shows that, for every $0\leq n\leq h$,
\begin{equation}
\bar{f}_{\tau}^{(n)} =  \sum_{k=1}^m\sum_{s=n+1}^{h+1}\left(\sum_{i\in\,\mathbb{Z}/p^{s}\mathbb{Z}}\frac{c_\gamma(k,s,i)}{\left(x-\zeta_{p^s}^i\gamma^{p^{h-s}}\right)^{k}} \right)+\sum_{k=1}^m\left(\sum_{i\in\,\mathbb{Z}/p^{n}\mathbb{Z}}\frac{\hat{c}_\gamma(k,n,i)}{\left(x-\zeta_{p^n}^i\gamma^{p^{h-n}}\right)^k}\right) \label{eq:residues-induction1}\end{equation}in case $\tau\not\subset\mathbb{K}^\times_t$; and\begin{multline}
\bar{f}_{\tau}^{(n)} =\sum_{k=1}^m\sum_{s=n+1}^{h+1}\left(\,\sideset{}{'}\sum_{i\in\,\mathbb{Z}/p^s\mathbb{Z}}\left(\sum_{\ell\in\,\mathbb{Z}/e\mathbb{Z}}\frac{d_\gamma(k,s,i,\ell)}{\left(x-\zeta_{p^s}^i\gamma^{p^{\ell-\pi_e(s)}}\right)^k}\right)\right)+ \hphantom{x}\\ +\sum_{k=1}^m\left(\,\sideset{}{'}\sum_{i\in\,\mathbb{Z}/p^n\mathbb{Z}}\left(\sum_{\ell\in\,\mathbb{Z}/e\mathbb{Z}}\frac{\hat{d}_\gamma(k,n,i,\ell)}{\left(x-\zeta_{p^n}^i\gamma^{p^{\ell-\pi_e(n)}}\right)^k}\right)\right)  \label{eq:residues-induction2}
\end{multline}
in case $\tau\subset\mathbb{K}^\times_t$. The harmless summand for $s=h+1$ in \eqref{eq:residues-induction1} and \eqref{eq:residues-induction2} is included so that the sums make sense for $n=h$, but we set every $c_\gamma(k,h+1,i):=0$ in \eqref{eq:residues-induction1} and every $d_\gamma(k,h+1,i,\ell):=0$ in \eqref{eq:residues-induction2}. The induction argument is straightforward, requiring only: the recursive definition of the coefficients $\hat{c}_\gamma(k,n,i)$ in \eqref{eq:non-torsion-residue-coefficients}, and the computation \eqref{eq:sigma-non-torsion}, in case $\tau\not\subset\mathbb{K}^\times_t$; the recursive definition of the coefficients $\hat{d}_\gamma(k,n,i,\ell)$ in \eqref{eq:torsion-residue-coefficient0} and \eqref{eq:torsion-residue-coefficients}, and the computations \eqref{eq:sigma-torsion0} and \eqref{eq:sigma-torsion+}, in case $\tau\subset\mathbb{K}^\times_t$; and a moderate amount of space and courage to write it down in detail in each case. It then follows from \eqref{eq:residues-induction1} and Definition~\ref{defn:non-torsion-residues} (in case $\tau\not\subset\mathbb{K}^\times_t$), or from \eqref{eq:residues-induction2} and Definition~\ref{defn:torsion-residues} (in case $\tau\subset\mathbb{K}^\times_t$), that \begin{equation}\label{eq:last-f-bar}\bar{f}_{\tau}^{(h)}=\sum_{k=1}^m\sum_{ \alpha\in \,\tau}\frac{\mathrm{dres}(f,\tau,k)_\alpha}{(x-\alpha)^k},
\end{equation} which holds uniformly in both cases $\tau\not\subset\mathbb{K}^\times_t$ and $\tau\subset\mathbb{K}^\times_t$. Also in both of these cases we have that $\bar{f}_\tau^{(h)}=f_\tau+\Delta(g_\tau^{(h)})$, and therefore $f_\tau$ is Mahler summable if and only if $\bar{f}_\tau^{(h)}$ is Mahler summable.

We claim that $\bar{f}_\tau^{(h)}$ is Mahler summable if and only if $\bar{f}_\tau^{(h)}=0$. This will establish the Proposition, since $\bar{f}_\tau^{(h)}=0$ if and only if $\mathrm{dres}(f,\tau,k)=\mathbf{0}$ for all $k\in\mathbb{N}$ by \eqref{eq:last-f-bar}. The non-trivial implication: $\bar{f}_\tau^{(h)}\neq0\Rightarrow \bar{f}_\tau^{(h)}$ is not Mahler summable, is proved in two cases: (1) if either $\tau\not\subset\mathbb{K}^\times_t$ or $h\neq 0$; and (2) if both $\tau\subset\mathbb{K}^\times_t$ and $h=0$.

(1). In case $\tau\not\subset\mathbb{K}^\times_t$, by Definition~\ref{defn:non-torsion-residues} $\bar{f}_\tau^{(h)}$ has no poles outside of $ \{\zeta_{p^h}^i\gamma \ | \ i\in\mathbb{Z}/p^h\mathbb{Z}\}$. In case $\tau\subset\mathbb{K}^\times_t$ and $h\neq 0$, by Definition~\ref{defn:torsion-residues} $\bar{f}_\tau^{(h)}$ has no poles outside of $\{\zeta_{p^h}^i\gamma^{p^\ell} \ | \ \ell\in\mathbb{Z}/e\mathbb{Z}, \ i\in (\mathbb{Z}/p^h\mathbb{Z})^\times\}$ (cf.~Lemma~\ref{lem:root-representation}). Thus if either $\tau\not\subset \mathbb{K}^\times_t$ or $h\neq 0$, $\mathrm{disp}(\bar{f}_\tau^{(h)},\tau)=0$. By Corollary~\ref{cor:summable-dispersion}, $\bar{f}_\tau^{(h)}$ is not Mahler summable.

(2). Note that $f_\tau=\bar{f}^{(h)}_\tau$ in \eqref{eq:last-f-bar} in this case where $h=0$ and $\tau\subset\mathbb{K}^\times_t$, and the Definition~\ref{defn:torsion-residues} gives \[f_\tau=\sum_{k=1}^m\sum_{\ell\in\mathbb{Z}/e\mathbb{Z}}\frac{d_\gamma(k,\ell)}{(x-\gamma^{p^\ell})^k}=\sum_{k=1}^m\sum_{\ell\in\mathbb{Z}/e\mathbb{Z}}\frac{\mathrm{dres}(f,\tau,k)_{\gamma^{p^\ell}}}{(x-\gamma^{p^\ell})^k},\] where we write $d_\gamma(k,\ell)$ in lieu of $d_\gamma(k,0,0,\ell)$, to simplify notation. Since $\tau\in\mathrm{supp}(f)$, we must have $\mathrm{dres}(f,\tau,m)\neq\mathbf{0}$. We claim that $f_\tau$ cannot be Mahler summable. To prove the claim, let again \[g_{\tau}^{(1)}:=\sum_{k=1}^m\sum_{\ell\in\,\mathbb{Z}/e\mathbb{Z}}\frac{c(k,\ell)}{\bigl(x-\gamma^{p^\ell}\bigr)^k}\] despite having $h=0$, where the $c_\gamma(k,\ell)$ are as in \eqref{eq:torsion-residue-coefficient-L}. By the computation~\eqref{eq:sigma-torsion0} and the Definition~\ref{defn:cyclic-coefficients} of the map $\mathcal{L}_\gamma^{(m)}$, \[\check{f}_\tau:=f_{\tau}+\Delta\left(g_{\tau}^{(1)}\right)=\sum_{k=1}^m\sum_{i=1}^{p-1}\sum_{\ell\in\,\mathbb{Z}/e\mathbb{Z}}\frac{\sum_{s=k}^m V^s_k\left(\zeta_p^i\gamma^{p^{\ell-1}}\right)\cdot c_\gamma(k,\ell)}{\left(x-\zeta_p^i\gamma^{p^{\ell-1}}\right)^k}.\] Hence $f_\tau$ is Mahler summable if and only if $\check{f}_\tau$ is Mahler summable. In particular, if $\check{f}_\tau=0$, then $f_\tau$ is Mahler summable. On the other hand, if $\check{f}_\tau\neq 0$, then $\mathrm{disp}(\check{f}_\tau,\tau)=0$, in which case $\check{f}_\tau$ cannot be Mahler summable by Corollary~\ref{cor:summable-dispersion}. Hence $f_\tau$ is Mahler summable if and only if $\check{f}_\tau=0$. Let us show that $f_\tau\neq 0\Rightarrow \check{f}_\tau\neq 0$.

In any case, the partial fraction coefficients of $\check{f}_\tau$ satisfy \[ \sum_{s=k}^m V^s_k\left(\zeta_p^i\gamma^{p^{\ell}}\right)\cdot c_\gamma(k,\ell+1)= \hphantom{x} = \zeta_p^{ik}\cdot\sum_{s=k}^m V^s_k\left(\gamma^{p^{\ell}}\right)\cdot c_\gamma(k,\ell+1) =\zeta_p^{ik}\cdot\bigl(c_\gamma(k,\ell)-d_\gamma(k,\ell)  \bigr), \] where the first equality follows from $\smash{V^s_k(\zeta_p^i\gamma^{p^\ell})=\zeta_p^{ik}V^s_k(\gamma^{p^\ell})}$ $\vphantom{\zeta_{p^\ell}}$ independently of $s$ by Lemma~\ref{lem:mahler-coefficients}, and the second equality follows from the Definition~\ref{defn:cyclic-coefficients} of $\mathcal{L}_\gamma^{(m)}$. By Lemma~\ref{lem:cyclic-coefficients}, since the map $\mathcal{D}_\gamma^{(m)}$ has no non-trivial fixed points, we cannot have $c_\gamma(k,\ell)=d_\gamma(k,\ell)$ for every $k$ and $\ell$ unless all $d_\gamma(k,\ell)=0$. So indeed $f_\tau\neq 0\Rightarrow \check{f}_\tau\neq 0$.
\end{proof}

\begin{rem}\label{rem:tree-remainder} For $f\in\mathbb{K}(x)$, $\tau\in\mathrm{supp}(f)\cap\mathcal{T}_M$, and $h:=h(f,\tau)$ as in Definition~\ref{defn:height}, the elements $\bar{f}_{\tau}^{(h)},g_{\tau}^{(h)}  \in\mathbb{K}(x)_{\tau}$ constructed in the proof of Proposition~\ref{prop:tree-residues} are the $\tau$-components of the $\bar{f},g\in\mathbb{K}(x)$ in the \emph{Mahler reduction} \eqref{eq:mahler-reduction}.
\end{rem}

\begin{proof}[Proof of the~\hypertarget{thm:main}{Main Theorem}] Let $f\in\mathbb{K}(x)$. By Lemma~\ref{lem:rational-decomposition}, $f$ is Mahler summable if and only if both $f_L$ and $f_T$ are Mahler summable. By Proposition~\ref{prop:laurent-residues}, $f_L$ is Mahler summable if and only if $\mathrm{dres}(f,\infty)=\mathbf{0}$. By Lemma~\ref{lem:tree-decomposition}, $f_T$ is Mahler summable if and only if $f_\tau$ is Mahler summable for all $\tau\in\mathcal{T}_M$. By Proposition~\ref{prop:tree-residues}, $f_\tau$ is Mahler summable if and only if $\mathrm{dres}(f,\tau,k)=\mathbf{0}$ for all $k\in\mathbb{N}$.\end{proof}


\subsection{Mahler reduction}\label{sec:reduction}

We can now define the Mahler reduction \eqref{eq:mahler-reduction}: $\bar{f}=f+\Delta(g)$ promised in the introduction for any $f\in\mathbb{K}(x)$, in terms of the decompositions $\bar{f}=\bar{f}_L+\bar{f}_T$ and $g=g_L+g_T$ as in \eqref{eq:f-decomposition}, by setting
\begin{align*}
\bar{f}_L:=\sum_{\theta\in\,\mathbb{Z}/\mathcal{P}}\bar{f}_{\theta}\qquad {}&\text{and}\qquad g_L:=\sum_{\theta\in\,\mathbb{Z}/\mathcal{P}}g_{\theta};\qquad \text{and}\\
\bar{f}_T:=\sum_{\tau\in\,\mathrm{supp}(f)}  \bar{f}_{\tau}^{\left(h(f,\tau)\right)} \qquad {}&\text{and}\qquad  g_T:=\sum_{\tau\in\,\mathrm{supp}(f)} g_{\tau}^{\left(h(f,\tau)\right)}\end{align*} as in Remark~\ref{rem:laurent-remainder} and Remark~\ref{rem:tree-remainder}. It is clear from the definitions that $\overline{c\cdot f}=c\cdot\bar{f}$ for $c\in\mathbb{K}$. Setting $\bar{f}_1 \,\tilde{+}\, \bar{f}_2:=\overline{f_1+f_2}$ defines a $\mathbb{K}$\nobreakdash-linear structure on $\{\bar{f} \ | \ f\in\mathbb{K}(x)\}$ such that $\nabla:f\mapsto \bar{f}$ is $\mathbb{K}$\nobreakdash-linear and  has the desired property that $\mathrm{ker}(\nabla)=\mathrm{im}(\Delta)$.


\section{Examples} Let us illustrate the Mahler discrete residues at Mahler trees in two small examples, with notation as in Example~\ref{eg:trees}. Example~\ref{EG:summable} gives a Mahler summable $f$ in the non-torsion case $\tau\not\subset\mathbb{K}^\times_t$. Example~\ref{EG:nonsummable} gives a non-Mahler summable $f$ in the torsion case $\tau\subset\mathbb{K}^\times_t$.

\begin{eg} \label{EG:summable}

Let $\tau=\tau(2)$, and consider the following $f=f_\tau$ with \[\mathrm{sing}(f,\tau)=\{2, \sqrt[3]{2}, \zeta_3\sqrt[3]{2}, \zeta_3^2\sqrt[3]{2}\}.\] By Definition~\ref{defn:height}, \mbox{$h=h(f,\tau)=1$.}
\begin{multline*}
f  = \frac{-x^6 + 4 x^3 + x^2 - 4 x}{(x - 2)^2 (x^3 - 2)^2}= \sum_{k=1}^2\sum_{n=0}^1\sum_{i=0}^{3^n-1}\frac{c_\gamma(k,n,i)}{\left(x-\zeta_{3^n}^i\sqrt[3]{2}^{3^{1-n}}\right)^k}\\
  =\frac{-1}{(x-2)^2}+\frac{1}{18\sqrt[3]{2}}\cdot\mathlarger{\sum_{i=0}^2 }\frac{\zeta_3^{2i}}{(x-\zeta_3^i\sqrt[3]{2})^2}-\frac{1}{9\sqrt[3]{4}}\cdot\mathlarger{\sum_{i=0}^2 }\frac{\zeta_3^i}{x-\zeta_3^i\sqrt[3]{2}},
\end{multline*}
for $\gamma:=\sqrt[3]{2}$ as in~Lemma~\ref{lem:non-torsion-expansion}. By Definition~\ref{defn:non-torsion-residues}, for $0\leq i\leq 2$:\begin{align*}
\mathrm{dres}(f,\tau,1)_{\zeta_3^i\gamma } &= \frac{-\zeta_3^i}{9\gamma^2}+V^2_1(\zeta_3^i\gamma )\cdot(-1)+V^1_1(\zeta_3^i\gamma )\cdot 0\\
&=\frac{-\zeta_3^i}{9\gamma^2}-\frac{2}{9}\cdot(\zeta_3^i\gamma)^{-5}=0; \qquad \text{and}\\
\mathrm{dres}(f,\tau,2)_{\zeta_3^i\gamma } &=\frac{\zeta_3^{2i}}{18\gamma }+V^2_2(\zeta_3^i\gamma )\cdot(-1)
=\frac{\zeta_3^{2i}}{18\gamma}-3^{-2}(\zeta_3^i\gamma )^{-4})=0;
\end{align*}
 by Lemma~\ref{lem:mahler-coefficients} and Example~\ref{eg:coefficients}. Therefore $f$ should be Mahler summable. And indeed, $f=\Delta((x-2)^{-2})$.\end{eg}
 
 \begin{eg} \label{EG:nonsummable} Let $\tau = \tau(\zeta_4)$, and consider the following $f=f_\tau$ with \[\mathrm{sing}(f,\tau)=\{\zeta_4, \ \zeta_4^3,\ \zeta_{12},\ \zeta_{12}^5,\ \zeta_{12}^{ 7},\ \zeta_{12}^{11}\}.\] By Definition~\ref{defn:height}, $h=h(f,\tau)=1$.
\[
  f = \dfrac{1}{x^6 + 1} \\
= \frac{1}{6} \left( \frac{\zeta_4^3}{x - \zeta_4} + \frac{\zeta_4}{x - \zeta_4^3} + \frac{\zeta_{12}^7}{x - \zeta_{12}} + \frac{\zeta_{12}^{11}}{x - \zeta_{12}^5} + \frac{\zeta_{12}}{x - \zeta_{12}^7} + \frac{\zeta_{12}^5}{x - \zeta_{12}^{11}} \right).\] The map $\mathcal{L}^{(1)}_{\zeta_4}$ in Definition~\ref{defn:cyclic-coefficients} sends $(\frac{\zeta_4^3}{6},\frac{\zeta_4}{6})\mapsto (\frac{\zeta_4^3}{4},\frac{\zeta_4}{4})$. By Definition~\ref{defn:torsion-residues}, for $1\leq i\leq 2$; $\ell\geq 1$; with $\alpha_{i,\ell}:=\zeta_3^i\zeta_4^{3^{\ell-1}}=\zeta_{12}^{4i+3^{\ell}}$:
\begin{multline*}
\mathrm{dres}(f,\tau,1)_{\alpha_{i,\ell}} =\frac{1}{6}\cdot\alpha_{i,\ell+1}+V^1_1(\alpha_{i,\ell})\cdot \frac{1}{4} \cdot \zeta_4^{3^{\ell - 1}} \\
= \frac{1}{6}\cdot \alpha_{i,\ell+1}+ \frac{1}{12} \cdot (\alpha_{i,\ell})^{-2} \cdot \zeta_4^{3^{\ell - 1}}
= \frac{1}{4} \cdot \alpha_{i,\ell+1} \neq 0,
\end{multline*} by Lemma~\ref{lem:mahler-coefficients}. Therefore $f$ is not Mahler summable. \end{eg}


\newcommand{\etalchar}[1]{$^{#1}$}

\end{document}